\documentclass[12pt]{amsart}

\usepackage[english]{babel}
\usepackage[margin=1in]{geometry}
\usepackage{amsmath}
\usepackage{amssymb}
\usepackage{amsfonts}
\usepackage{graphicx}
\usepackage{xcolor}
\usepackage[hidelinks]{hyperref}
\usepackage{tikz}
\usepackage[shortlabels]{enumitem}

\theoremstyle{plain}
\newtheorem{theorem}{Theorem}
\newtheorem{corollary}[theorem]{Corollary}
\newtheorem{proposition}[theorem]{Proposition}
\newtheorem{lemma}[theorem]{Lemma}

\theoremstyle{definition}

\newtheorem{remark}[theorem]{Remark}
\usetikzlibrary{cd}

\newcommand{\Y}{\mathcal{S}_q^Y}
\newcommand{\Sq}{\mathcal{S}_q}

\def\positivecrossing{\tikz[baseline=0]{
		\draw[dashed] (0,0) circle (1);
		\draw[thick] (-.7,-.7) -- (.7,.7);
		\draw[thick] (.1,-.1) -- (.7,-.7);
		\draw[thick] (-.7,.7) -- (-.1,.1);
}}

\def\verticalresolution{\tikz[baseline=0]{
		\draw[dashed] (0,0) circle (1);
		\draw[thick] plot [smooth, tension=1] coordinates {(-.7,.7) (-.3,0) (-.7,-.7)};
		\draw[thick] plot [smooth, tension=1] coordinates {(.7,.7) (.3,0) (.7,-.7)};
}}
\def\horizontalresolution{\tikz[baseline=0]{
		\draw[dashed] (0,0) circle (1);
		\draw[thick] plot [smooth, tension=1] coordinates {(-.7,.7) (0,.3) (.7,.7)};
		\draw[thick] plot [smooth, tension=1] coordinates {(-.7,-.7) (0,-.3) (.7,-.7)};
}}
\def\trivialcomponent{\tikz[baseline=0]{
		\draw[dashed] (0,0) circle (1);
		\draw[thick] (0,0) circle (.5);
}}
\def\htangle{\tikz[baseline=0]{
		\draw[dashed] (0,0) circle (1);
		\draw[thick] (-.4,0) -- (.4,0);
		\draw[thick] (-.8,.6) -- (-.4,0);
		\draw[thick] (-.8,-.6) -- (-.4,0);
		\draw[thick] (.8,.6) -- (.4,0);
		\draw[thick] (.8,-.6) -- (.4,0);
}}
\def\itangle{\tikz[baseline=0]{
		\draw[dashed] (0,0) circle (1);
		\draw[thick] (0,.4) -- (0,-.4);
		\draw[thick] (-.6,.8) -- (0,.4);
		\draw[thick] (.6,.8) -- (0,.4);
		\draw[thick] (-.6,-.8) -- (0,-.4);
		\draw[thick] (.6,-.8) -- (0,-.4);
}}
\def\monogon{\tikz[baseline=0]{
		\draw[dashed] (0,0) circle (1);
		\draw[thick] (0,0) circle (.5);
		\draw[thick] (0,-.5) -- (0,-1);
}}
\def\bigon{\tikz[baseline=0]{
		\draw[dashed] (0,0) circle (1);
		\draw[thick] (-1,0) -- (-.5,0);
		\draw[thick] (.5,0) -- (1,0);
		\draw[thick] plot [smooth, tension = 1] coordinates {(-.5,0) (0,.25) (.5,0)};
		\draw[thick] plot [smooth, tension = 1] coordinates {(-.5,0) (0,-.25) (.5,0)};
}}
\def\trigon{\tikz[baseline=0]{
		\draw[dashed] (0,0) circle (1);
		\draw[thick] (0,.5) -- (-.433,-.25);
		\draw[thick] (0,.5) -- (.433,-.25);
		\draw[thick] (-.433,-.25) -- (.433,-.25);
		\draw[thick] (0,.5) -- (0,1);
		\draw[thick] (-.433,-.25) -- (-.866,-.5);
		\draw[thick] (.433,-.25) -- (.866,-.5);
}}
\def\strand{\tikz[baseline=0]{
		\draw[dashed] (0,0) circle (1);
		\draw[thick] (-1,0) -- (1,0);
}}
\def\jwstrand{\tikz[baseline=0]{
		\draw[dashed] (0,0) circle (1);
		\draw[thick] (-1,.1) -- (-.2,.1);
		\draw[thick] (.2,.1) -- (1,.1);
		\draw[thick] (-1,-.1) -- (-.2,-.1);
		\draw[thick] (.2,-.1) -- (1,-.1);
		\draw[thick] (-.2,-.2) rectangle (.2,.2);
}}
\def\jwidentity{\tikz[baseline=0]{
		\draw[dashed] (0,0) circle (1);
		\draw[thick] (-1,.1) -- (1,.1);
		\draw[thick] (-1,-.1) -- (1,-.1);
}}
\def\jwreturn{\tikz[baseline=0]{
		\draw[dashed] (0,0) circle (1);
		\draw[thick] (-1,.1) -- (-.2,.1);
		\draw[thick] (.2,.1) -- (1,.1);
		\draw[thick] (-1,-.1) -- (-.2,-.1);
		\draw[thick] (.2,-.1) -- (1,-.1);
		\draw[thick] (-.2,-.1) arc[start angle = -90, end angle = 90, radius = .1];
		\draw[thick] (.2,.1) arc[start angle = 90, end angle = 270, radius = .1];
}}
\def\vertex{\tikz[baseline=0]{
		\draw[dashed] (0,0) circle (1);
		\draw[thick] (0,0) -- (0,1);
		\draw[thick] (0,0) -- (-1.71 * .5,-.5);
		\draw[thick] (0,0) -- (1.71 * .5,-.5);
}}
\def\jwvertex{\tikz[baseline=0]{
		\draw[dashed] (0,0) circle (1);
		\draw[thick] (.1,1) to [bend right = 25] (1.71 * .5 + .05,-.4);
		\draw[thick] (0,1) to [bend left = 25] (-1.71 * .5 - .05,-.4);
		\draw[thick] (1.71 * .5,-.5) to [bend right = 25] (-1.71 * .5,-.5);
}}

\def\pqknot{\tikz[baseline=0]{
		\node at (0,1.5) {$k$ \text{times}};
		\draw[-stealth] (0,1.35) -- (-.6,1.1);
		\draw[-stealth] (0,1.35) -- (-.2,1.1);
		\draw[-stealth] (0,1.35) -- (.2,1.1);
		\draw[-stealth] (0,1.35) -- (.6,1.1);
		\draw[-stealth] (0,1.35) -- (1,1.1);
		\node at (-2.1,0) {$l$ times};
		\draw[-stealth] (-1.35,0) -- (-1.1, -1);
		\draw[-stealth] (-1.35,0) -- (-1.1, 0);
		\draw[->>] (-1,-1) -- (-1,.1);
		\draw (-1,.1) -- (-1,1);
		\draw[->>] (1,-1) -- (1,.1);
		\draw (1,.1) -- (1,1);
		\draw[->] (-1,1) -- (.1,1);
		\draw (.1,1) -- (1,1);
		\draw[->] (-1,-1) -- (.1,-1);
		\draw (.1,-1) -- (1,-1);
		\draw[thick, -stealth] (-1,-1) -- (-.6,0);
		\draw[thick] (-.6,0) -- (-.2,1);
		\draw[thick] (-.2,-1) -- (.6,1);
		\draw[thick] (.6,-1) -- (1,0);
		\draw[thick] (-1,0) -- (-.6,1);
		\draw[thick] (-.6,-1) -- (.2,1);
		\draw[thick] (.2,-1) -- (1,1);
}}

\def\hexagon{\tikz[baseline=0]{
		\draw[->>] (-1,-1) -- (-1,.1);
		\draw (-1,.1) -- (-1,1);
		\draw[->>] (1,-1) -- (1,.1);
		\draw (1,.1) -- (1,1);
		\draw[->] (-1,1) -- (.1,1);
		\draw (.1,1) -- (1,1);
		\draw[->] (-1,-1) -- (.1,-1);
		\draw (.1,-1) -- (1,-1);
		\draw[thick] (-1,0) -- (1,0);
		\draw[thick] (-.2,0) -- (0,1);
		\draw[thick] (0,-1) -- (.2,0);
}}
\def\ranktwomonogon{\tikz[baseline=0]{
		\draw[->>] (-1,-1) -- (-1,.1);
		\draw (-1,.1) -- (-1,1);
		\draw[->>] (1,-1) -- (1,.1);
		\draw (1,.1) -- (1,1);
		\draw[->] (-1,1) -- (.1,1);
		\draw (.1,1) -- (1,1);
		\draw[->] (-1,-1) -- (.1,-1);
		\draw (.1,-1) -- (1,-1);
		\draw[thick] (-1,0) -- (1,0);
		\draw[thick] (-.2,0) -- (0,1);
		\draw[thick] (0,-1) -- (.2,0);
		\draw[thick] (.4,0) -- (.4,.2);
		\draw[thick] (.8,0) -- (.8,.2);
		\draw[thick] (.4,.3) circle (.1);
		\draw[thick] (.8,.3) circle (.1);
		\draw[fill] (.55,.3) circle (.01);
		\draw[fill] (.6,.3) circle (.01);
		\draw[fill] (.65,.3) circle (.01);
		\node at (.6,.5) {$\overbrace{\begin{tikzpicture} \path (-.3,.3); \end{tikzpicture}}^{\tiny{n - 1}}$};
}}

\def\jwktilde{\tikz[baseline=0]{
		\draw[->>] (-1,-1) -- (-1,.1);
		\draw (-1,.1) -- (-1,1);
		\draw[->>] (1,-1) -- (1,.1);
		\draw (1,.1) -- (1,1);
		\draw[->] (-1,1) -- (.1,1);
		\draw (.1,1) -- (1,1);
		\draw[->] (-1,-1) -- (.1,-1);
		\draw (.1,-1) -- (1,-1);
		\draw (.5,.5) circle (.1);
		\draw[thick] (-1,.1) -- (-.9,.1);
		\draw[thick] (-1,-.1) -- (-.9,-.1);
		\draw[thick] (-.6,.1) -- (-.3,.1);
		\draw[thick] (-.6,-.1) -- (.1,-.1);
		\draw[thick] (-.1,.1) -- (.1,.1);
		\draw[thick] (.4,.1) -- (1,.1);
		\draw[thick] (.4,-.1) -- (.7,-.1);
		\draw[thick] (.9,-.1) -- (1,-.1);
		\draw[thick] (-.3,.1) -- (-.3,.3);
		\draw[thick] (-.1,.1) -- (-.1,.3);
		\draw[thick] (-.3,.6) -- (0,1);
		\draw[thick] (-.1,.6) -- (.2,1);
		\draw[thick] (.7,-.1) -- (0,-1);
		\draw[thick] (.9,-.1) -- (.2,-1);
		\draw[thick] (-.9, .15) rectangle (-.6, -.15);
		\draw[thick] (-.35,.3) rectangle (-.05,.6);
		\draw[thick] (.1,.15) rectangle (.4,-.15);
}}
\def\jwktildezero{\tikz[baseline=0]{
		\draw[->>] (-1,-1) -- (-1,.1);
		\draw (-1,.1) -- (-1,1);
		\draw[->>] (1,-1) -- (1,.1);
		\draw (1,.1) -- (1,1);
		\draw[->] (-1,1) -- (.1,1);
		\draw (.1,1) -- (1,1);
		\draw[->] (-1,-1) -- (.1,-1);
		\draw (.1,-1) -- (1,-1);
		\draw (.5,.5) circle (.1);
		\draw[thick] (-1,.1) -- (-.9,.1);
		\draw[thick] (-1,-.1) -- (-.9,-.1);
		\draw[thick] (-.6,.1) -- (-.3,.1);
		\draw[thick] (-.6,-.1) -- (.1,-.1);
		\draw[thick] (-.1,.1) -- (.1,.1);
		\draw[thick] (.4,.1) -- (1,.1);
		\draw[thick] (.4,-.1) -- (.7,-.1);
		\draw[thick] (.9,-.1) -- (1,-.1);
		\draw[thick] (-.3,.1) -- (-.3,.3);
		\draw[thick] (-.1,.1) -- (-.1,.3);
		\draw[thick] (-.3,.6) -- (0,1);
		\draw[thick] (-.1,.6) -- (.2,1);
		\draw[thick] (.7,-.1) -- (0,-1);
		\draw[thick] (.9,-.1) -- (.2,-1);
		\draw[thick] (-.9,.1) -- (-.6,.1);
		\draw[thick] (-.9,-.1) -- (-.6,-.1);
		\draw[thick] (-.3,.3) -- (-.3,.6);
		\draw[thick] (-.1,.3) -- (-.1,.6);
		\draw[thick] (.1,.1) -- (.4,.1);
		\draw[thick] (.1,-.1) -- (.4,-.1);
}}
\def\jwktildeone{\tikz[baseline=0]{
		\draw[->>] (-1,-1) -- (-1,.1);
		\draw (-1,.1) -- (-1,1);
		\draw[->>] (1,-1) -- (1,.1);
		\draw (1,.1) -- (1,1);
		\draw[->] (-1,1) -- (.1,1);
		\draw (.1,1) -- (1,1);
		\draw[->] (-1,-1) -- (.1,-1);
		\draw (.1,-1) -- (1,-1);
		\draw (.5,.5) circle (.1);
		\draw[thick] (-1,.1) -- (-.9,.1);
		\draw[thick] (-1,-.1) -- (-.9,-.1);
		\draw[thick] (-.6,.1) -- (-.3,.1);
		\draw[thick] (-.6,-.1) -- (.1,-.1);
		\draw[thick] (-.1,.1) -- (.1,.1);
		\draw[thick] (.4,.1) -- (1,.1);
		\draw[thick] (.4,-.1) -- (.7,-.1);
		\draw[thick] (.9,-.1) -- (1,-.1);
		\draw[thick] (-.3,.1) -- (-.3,.3);
		\draw[thick] (-.1,.1) -- (-.1,.3);
		\draw[thick] (-.3,.6) -- (0,1);
		\draw[thick] (-.1,.6) -- (.2,1);
		\draw[thick] (.7,-.1) -- (0,-1);
		\draw[thick] (.9,-.1) -- (.2,-1);
		\draw[thick] (-.9,.1) -- (-.6,.1);
		\draw[thick] (-.9,-.1) -- (-.6,-.1);
		\draw[thick] (-.3,.3) -- (-.3,.6);
		\draw[thick] (-.1,.3) -- (-.1,.6);
		\draw[thick] (.1,.1) arc[start angle = 90, end angle = -90, radius = .1];
		\draw[thick] (.4,.1) arc[start angle = 90, end angle = 270, radius = .1];
}}
\def\jwktildetwo{\tikz[baseline=0]{
		\draw[->>] (-1,-1) -- (-1,.1);
		\draw (-1,.1) -- (-1,1);
		\draw[->>] (1,-1) -- (1,.1);
		\draw (1,.1) -- (1,1);
		\draw[->] (-1,1) -- (.1,1);
		\draw (.1,1) -- (1,1);
		\draw[->] (-1,-1) -- (.1,-1);
		\draw (.1,-1) -- (1,-1);
		\draw (.5,.5) circle (.1);
		\draw[thick] (-1,.1) -- (-.9,.1);
		\draw[thick] (-1,-.1) -- (-.9,-.1);
		\draw[thick] (-.6,.1) -- (-.3,.1);
		\draw[thick] (-.6,-.1) -- (.1,-.1);
		\draw[thick] (-.1,.1) -- (.1,.1);
		\draw[thick] (.4,.1) -- (1,.1);
		\draw[thick] (.4,-.1) -- (.7,-.1);
		\draw[thick] (.9,-.1) -- (1,-.1);
		\draw[thick] (-.3,.1) -- (-.3,.3);
		\draw[thick] (-.1,.1) -- (-.1,.3);
		\draw[thick] (-.3,.6) -- (0,1);
		\draw[thick] (-.1,.6) -- (.2,1);
		\draw[thick] (.7,-.1) -- (0,-1);
		\draw[thick] (.9,-.1) -- (.2,-1);
		\draw[thick] (-.9,.1) -- (-.6,.1);
		\draw[thick] (-.9,-.1) -- (-.6,-.1);
		\draw[thick] (-.3,.3) arc[start angle = 180, end angle = 0, radius = .1];
		\draw[thick] (-.3,.6) arc[start angle = 180, end angle = 360, radius = .1];
		\draw[thick] (.1,.1) -- (.4,.1);
		\draw[thick] (.1,-.1) -- (.4,-.1);
}}
\def\jwktildethree{\tikz[baseline=0]{
		\draw[->>] (-1,-1) -- (-1,.1);
		\draw (-1,.1) -- (-1,1);
		\draw[->>] (1,-1) -- (1,.1);
		\draw (1,.1) -- (1,1);
		\draw[->] (-1,1) -- (.1,1);
		\draw (.1,1) -- (1,1);
		\draw[->] (-1,-1) -- (.1,-1);
		\draw (.1,-1) -- (1,-1);
		\draw (.5,.5) circle (.1);
		\draw[thick] (-1,.1) -- (-.9,.1);
		\draw[thick] (-1,-.1) -- (-.9,-.1);
		\draw[thick] (-.6,.1) -- (-.3,.1);
		\draw[thick] (-.6,-.1) -- (.1,-.1);
		\draw[thick] (-.1,.1) -- (.1,.1);
		\draw[thick] (.4,.1) -- (1,.1);
		\draw[thick] (.4,-.1) -- (.7,-.1);
		\draw[thick] (.9,-.1) -- (1,-.1);
		\draw[thick] (-.3,.1) -- (-.3,.3);
		\draw[thick] (-.1,.1) -- (-.1,.3);
		\draw[thick] (-.3,.6) -- (0,1);
		\draw[thick] (-.1,.6) -- (.2,1);
		\draw[thick] (.7,-.1) -- (0,-1);
		\draw[thick] (.9,-.1) -- (.2,-1);
		\draw[thick] (-.9,.1) -- (-.6,.1);
		\draw[thick] (-.9,-.1) -- (-.6,-.1);
		\draw[thick] (-.3,.3) arc[start angle = 180, end angle = 0, radius = .1];
		\draw[thick] (-.3,.6) arc[start angle = 180, end angle = 360, radius = .1];
		\draw[thick] (.1,.1) arc[start angle = 90, end angle = -90, radius = .1];
		\draw[thick] (.4,.1) arc[start angle = 90, end angle = 270, radius = .1];
}}
\def\jwktildefour{\tikz[baseline=0]{
		\draw[->>] (-1,-1) -- (-1,.1);
		\draw (-1,.1) -- (-1,1);
		\draw[->>] (1,-1) -- (1,.1);
		\draw (1,.1) -- (1,1);
		\draw[->] (-1,1) -- (.1,1);
		\draw (.1,1) -- (1,1);
		\draw[->] (-1,-1) -- (.1,-1);
		\draw (.1,-1) -- (1,-1);
		\draw (.5,.5) circle (.1);
		\draw[thick] (-1,.1) -- (-.9,.1);
		\draw[thick] (-1,-.1) -- (-.9,-.1);
		\draw[thick] (-.6,.1) -- (-.3,.1);
		\draw[thick] (-.6,-.1) -- (.1,-.1);
		\draw[thick] (-.1,.1) -- (.1,.1);
		\draw[thick] (.4,.1) -- (1,.1);
		\draw[thick] (.4,-.1) -- (.7,-.1);
		\draw[thick] (.9,-.1) -- (1,-.1);
		\draw[thick] (-.3,.1) -- (-.3,.3);
		\draw[thick] (-.1,.1) -- (-.1,.3);
		\draw[thick] (-.3,.6) -- (0,1);
		\draw[thick] (-.1,.6) -- (.2,1);
		\draw[thick] (.7,-.1) -- (0,-1);
		\draw[thick] (.9,-.1) -- (.2,-1);
		\draw[thick] (-.9,.1) arc[start angle = 90, end angle = -90, radius = .1];
		\draw[thick] (-.6,.1) arc[start angle = 90, end angle = 270, radius = .1];
		\draw[thick] (-.3,.3) -- (-.3,.6);
		\draw[thick] (-.1,.3) -- (-.1,.6);
		\draw[thick] (.1,.1) -- (.4,.1);
		\draw[thick] (.1,-.1) -- (.4,-.1);
}}
\def\jwktildefive{\tikz[baseline=0]{
		\draw[->>] (-1,-1) -- (-1,.1);
		\draw (-1,.1) -- (-1,1);
		\draw[->>] (1,-1) -- (1,.1);
		\draw (1,.1) -- (1,1);
		\draw[->] (-1,1) -- (.1,1);
		\draw (.1,1) -- (1,1);
		\draw[->] (-1,-1) -- (.1,-1);
		\draw (.1,-1) -- (1,-1);
		\draw (.5,.5) circle (.1);
		\draw[thick] (-1,.1) -- (-.9,.1);
		\draw[thick] (-1,-.1) -- (-.9,-.1);
		\draw[thick] (-.6,.1) -- (-.3,.1);
		\draw[thick] (-.6,-.1) -- (.1,-.1);
		\draw[thick] (-.1,.1) -- (.1,.1);
		\draw[thick] (.4,.1) -- (1,.1);
		\draw[thick] (.4,-.1) -- (.7,-.1);
		\draw[thick] (.9,-.1) -- (1,-.1);
		\draw[thick] (-.3,.1) -- (-.3,.3);
		\draw[thick] (-.1,.1) -- (-.1,.3);
		\draw[thick] (-.3,.6) -- (0,1);
		\draw[thick] (-.1,.6) -- (.2,1);
		\draw[thick] (.7,-.1) -- (0,-1);
		\draw[thick] (.9,-.1) -- (.2,-1);
		\draw[thick] (-.9,.1) arc[start angle = 90, end angle = -90, radius = .1];
		\draw[thick] (-.6,.1) arc[start angle = 90, end angle = 270, radius = .1];
		\draw[thick] (-.3,.3) -- (-.3,.6);
		\draw[thick] (-.1,.3) -- (-.1,.6);
		\draw[thick] (.1,.1) arc[start angle = 90, end angle = -90, radius = .1];
		\draw[thick] (.4,.1) arc[start angle = 90, end angle = 270, radius = .1];
}}
\def\jwktildesix{\tikz[baseline=0]{
		\draw[->>] (-1,-1) -- (-1,.1);
		\draw (-1,.1) -- (-1,1);
		\draw[->>] (1,-1) -- (1,.1);
		\draw (1,.1) -- (1,1);
		\draw[->] (-1,1) -- (.1,1);
		\draw (.1,1) -- (1,1);
		\draw[->] (-1,-1) -- (.1,-1);
		\draw (.1,-1) -- (1,-1);
		\draw (.5,.5) circle (.1);
		\draw[thick] (-1,.1) -- (-.9,.1);
		\draw[thick] (-1,-.1) -- (-.9,-.1);
		\draw[thick] (-.6,.1) -- (-.3,.1);
		\draw[thick] (-.6,-.1) -- (.1,-.1);
		\draw[thick] (-.1,.1) -- (.1,.1);
		\draw[thick] (.4,.1) -- (1,.1);
		\draw[thick] (.4,-.1) -- (.7,-.1);
		\draw[thick] (.9,-.1) -- (1,-.1);
		\draw[thick] (-.3,.1) -- (-.3,.3);
		\draw[thick] (-.1,.1) -- (-.1,.3);
		\draw[thick] (-.3,.6) -- (0,1);
		\draw[thick] (-.1,.6) -- (.2,1);
		\draw[thick] (.7,-.1) -- (0,-1);
		\draw[thick] (.9,-.1) -- (.2,-1);
		\draw[thick] (-.9,.1) arc[start angle = 90, end angle = -90, radius = .1];
		\draw[thick] (-.6,.1) arc[start angle = 90, end angle = 270, radius = .1];
		\draw[thick] (-.3,.3) arc[start angle = 180, end angle = 0, radius = .1];
		\draw[thick] (-.3,.6) arc[start angle = 180, end angle = 360, radius = .1];
		\draw[thick] (.1,.1) -- (.4,.1);
		\draw[thick] (.1,-.1) -- (.4,-.1);
}}
\def\jwktildeseven{\tikz[baseline=0]{
		\draw[->>] (-1,-1) -- (-1,.1);
		\draw (-1,.1) -- (-1,1);
		\draw[->>] (1,-1) -- (1,.1);
		\draw (1,.1) -- (1,1);
		\draw[->] (-1,1) -- (.1,1);
		\draw (.1,1) -- (1,1);
		\draw[->] (-1,-1) -- (.1,-1);
		\draw (.1,-1) -- (1,-1);
		\draw (.5,.5) circle (.1);
		\draw[thick] (-1,.1) -- (-.9,.1);
		\draw[thick] (-1,-.1) -- (-.9,-.1);
		\draw[thick] (-.6,.1) -- (-.3,.1);
		\draw[thick] (-.6,-.1) -- (.1,-.1);
		\draw[thick] (-.1,.1) -- (.1,.1);
		\draw[thick] (.4,.1) -- (1,.1);
		\draw[thick] (.4,-.1) -- (.7,-.1);
		\draw[thick] (.9,-.1) -- (1,-.1);
		\draw[thick] (-.3,.1) -- (-.3,.3);
		\draw[thick] (-.1,.1) -- (-.1,.3);
		\draw[thick] (-.3,.6) -- (0,1);
		\draw[thick] (-.1,.6) -- (.2,1);
		\draw[thick] (.7,-.1) -- (0,-1);
		\draw[thick] (.9,-.1) -- (.2,-1);
		\draw[thick] (-.9,.1) arc[start angle = 90, end angle = -90, radius = .1];
		\draw[thick] (-.6,.1) arc[start angle = 90, end angle = 270, radius = .1];
		\draw[thick] (-.3,.3) arc[start angle = 180, end angle = 0, radius = .1];
		\draw[thick] (-.3,.6) arc[start angle = 180, end angle = 360, radius = .1];
		\draw[thick] (.1,.1) arc[start angle = 90, end angle = -90, radius = .1];
		\draw[thick] (.4,.1) arc[start angle = 90, end angle = 270, radius = .1];
}}

\def\hexpunctured{\tikz[baseline=0]{
		\draw[->>] (-1,-1) -- (-1,.1);
		\draw (-1,.1) -- (-1,1);
		\draw[->>] (1,-1) -- (1,.1);
		\draw (1,.1) -- (1,1);
		\draw[->] (-1,1) -- (.1,1);
		\draw (.1,1) -- (1,1);
		\draw[->] (-1,-1) -- (.1,-1);
		\draw (.1,-1) -- (1,-1);
		\draw (.5,.5) circle (.1);
		\draw[thick] (-1,0) -- (1,0);
		\draw[thick] (-.2,0) -- (0,1);
		\draw[thick] (0,-1) -- (.2,0);
}}
\def\ranktwomonogonpunctured{\tikz[baseline=0]{
		\draw[->>] (-1,-1) -- (-1,.1);
		\draw (-1,.1) -- (-1,1);
		\draw[->>] (1,-1) -- (1,.1);
		\draw (1,.1) -- (1,1);
		\draw[->] (-1,1) -- (.1,1);
		\draw (.1,1) -- (1,1);
		\draw[->] (-1,-1) -- (.1,-1);
		\draw (.1,-1) -- (1,-1);
		\draw[thick] (-1,0) -- (1,0);
		\draw[thick] (-.2,0) -- (0,1);
		\draw[thick] (0,-1) -- (.06,-.7);
		\draw (.1,-.5) circle (.1);
		\draw[thick] (.1,-.5) circle (.2);
		\draw[thick] (.14,-.3) -- (.2,0);
		\draw[thick] (.4,0) -- (.4,.2);
		\draw[thick] (.8,0) -- (.8,.2);
		\draw[thick] (.4,.3) circle (.1);
		\draw[thick] (.8,.3) circle (.1);
		\draw[fill] (.55,.3) circle (.01);
		\draw[fill] (.6,.3) circle (.01);
		\draw[fill] (.65,.3) circle (.01);
		\node at (.6,.5) {$\overbrace{\begin{tikzpicture} \path (-.3,.3); \end{tikzpicture}}^{\tiny{n - 2}}$};
}}

\def\modifiedih{\tikz[baseline=0]{
		\draw[dashed] (0,0) circle (1);
		\draw (0,0) circle (.1);
		\draw[thick] (0,0) circle (.2);
		\draw[thick] (-.6,0) -- (-.2,0);
		\draw[thick] (.2,0) -- (.6,0);
		\draw[thick] (-.6,0) -- (-.9,.3);
		\draw[thick] (-.6,0) -- (-.9,-.3);
		\draw[thick] (.6,0) -- (.9,.3);
		\draw[thick] (.6,0) -- (.9,-.3);
		\draw[dotted] (-.4,0) ellipse (.3 and .2);
}}
\def\modifiedihi{\tikz[baseline=0]{
		\draw[dashed] (0,0) circle (1);
		\draw (0,0) circle (.1);
		\draw[thick] (0,0) circle (.2);
		\draw[thick] (-.141,.141) -- (-.9,.3);
		\draw[thick] (-.141,-.141) -- (-.9,-.3);
		\draw[thick] (.2,0) -- (.6,0);
		\draw[thick] (.6,0) -- (.9,.3);
		\draw[thick] (.6,0) -- (.9,-.3);
		\draw[dotted] (.4,0) ellipse (.3 and .2);
}}
\def\modifiedihii{\tikz[baseline=0]{
		\draw[dashed] (0,0) circle (1);
		\draw (0,0) circle (.1);
		\draw[thick] (0,0) circle (.2);
		\draw[thick] (-.141,.141) -- (-.9,.3);
		\draw[thick] (-.141,-.141) -- (-.9,-.3);
		\draw[thick] (.141,.141) -- (.9,.3);
		\draw[thick] (.141,-.141) -- (.9,-.3);
		\draw[dotted] (0,.15) ellipse (.2 and .1);
}}
\def\modifiedihiii{\tikz[baseline=0]{
		\draw[dashed] (0,0) circle (1);
		\draw (0,0) circle (.1);
		\draw[thick] (0,0) circle (.2);
		\draw[thick] (0,.2) -- (0,.3);
		\draw[thick] (-.9,.3) -- (.9,.3);
		\draw[thick] (-.141,-.141) -- (-.9,-.3);
		\draw[thick] (.141,-.141) -- (.9,-.3);
		\draw[dotted] (0,-.15) ellipse (.2 and .1);
}}
\def\modifiedihiv{\tikz[baseline=0]{
		\draw[dashed] (0,0) circle (1);
		\draw (0,0) circle (.1);
		\draw[thick] (0,0) circle (.2);
		\draw[thick] (0,.2) -- (0,.3);
		\draw[thick] (-.9,.3) -- (.9,.3);
		\draw[thick] (0,-.2) -- (0,-.3);
		\draw[thick] (-.9,-.3) -- (.9,-.3);
}}
\def\bigonhexpuncturedi{\tikz[baseline=0]{
		\draw[->>] (-1,-1) -- (-1,.1);
		\draw (-1,.1) -- (-1,1);
		\draw[->>] (1,-1) -- (1,.1);
		\draw (1,.1) -- (1,1);
		\draw[->] (-1,1) -- (.1,1);
		\draw (.1,1) -- (1,1);
		\draw[->] (-1,-1) -- (.1,-1);
		\draw (.1,-1) -- (1,-1);
		\draw[thick] (-1,0) -- (1,0);
		\draw[thick] (-.2,0) -- (0,1);
		\draw[thick] (0,-1) -- (.06,-.7);
		\draw (.1,-.5) circle (.1);
		\draw[thick] (.1,-.5) circle (.2);
		\draw[thick] (.14,-.3) -- (.2,0);
		\draw[dotted] (.17,-.15) ellipse (.1 and .2);
}}
\def\bigonhexpuncturedii{\tikz[baseline=0]{
		\draw[->>] (-1,-1) -- (-1,.1);
		\draw (-1,.1) -- (-1,1);
		\draw[->>] (1,-1) -- (1,.1);
		\draw (1,.1) -- (1,1);
		\draw[->] (-1,1) -- (.1,1);
		\draw (.1,1) -- (1,1);
		\draw[->] (-1,-1) -- (.1,-1);
		\draw (.1,-1) -- (1,-1);
		\draw (.2,0) circle (.1);
		\draw[thick] (.2,0) circle (.2);
		\draw[thick] (-1,0) -- (0,0);
		\draw[thick] (.4,0) -- (1,0);
		\draw[thick] (-.2,0) -- (0,1);
		\draw[thick] (.2,-.2) -- (0,-1);
		\draw[dotted] plot [smooth, tension=1] coordinates {(-.1,1) (-.22,-.1) (.1,1)};
		\draw[dotted] plot [smooth, tension=1] coordinates {((-.1,-1) (.22,-.1) (.1,-1)};
}}
\def\bigonhexpuncturediii{\tikz[baseline=0]{
		\draw[->>] (-1,-1) -- (-1,.1);
		\draw (-1,.1) -- (-1,1);
		\draw[->>] (1,-1) -- (1,.1);
		\draw (1,.1) -- (1,1);
		\draw[->] (-1,1) -- (.1,1);
		\draw (.1,1) -- (1,1);
		\draw[->] (-1,-1) -- (.1,-1);
		\draw (.1,-1) -- (1,-1);
		\draw (.25,-.4) circle (.1);
		\draw[thick] (-1,0) -- (1,0);
		\draw[thick] (-.6,0) -- (-.2,1);
		\draw[thick] (-.2,0) -- (.2,1);
		\draw[thick] (-.2,-1) -- (.2,0);
		\draw[thick] (.2,-1) -- (.6,0);
		\draw[dotted] (0,0) ellipse (.3 and .2);
}}
\def\bigonhexpuncturediv{\tikz[baseline=0]{
		\draw[->>] (-1,-1) -- (-1,.1);
		\draw (-1,.1) -- (-1,1);
		\draw[->>] (1,-1) -- (1,.1);
		\draw (1,.1) -- (1,1);
		\draw[->] (-1,1) -- (.1,1);
		\draw (.1,1) -- (1,1);
		\draw[->] (-1,-1) -- (.1,-1);
		\draw (.1,-1) -- (1,-1);
		\draw (.15,-.4) circle (.1);
		\draw[thick] (-1,0) -- (1,0);
		\draw[thick] (-.6,0) -- (-.4,1);
		\draw[thick] (-.4,-1) -- (-.2,0);
		\draw[thick] (.2,0) -- (.4,1);
		\draw[thick] (.4,-1) -- (.6,0);
		\draw[dotted] plot [smooth, tension=1] coordinates {(-1,.1) (-.5,0) (-1,-.1)};
		\draw[dotted] plot [smooth, tension=1] coordinates {(1,.1) (.5,0) (1,-.1)};
}}
\def\bigonhexpuncturedv{\tikz[baseline=0]{
		\draw[->>] (-1,-1) -- (-1,.1);
		\draw (-1,.1) -- (-1,1);
		\draw[->>] (1,-1) -- (1,.1);
		\draw (1,.1) -- (1,1);
		\draw[->] (-1,1) -- (.1,1);
		\draw (.1,1) -- (1,1);
		\draw[->] (-1,-1) -- (.1,-1);
		\draw (.1,-1) -- (1,-1);
		\draw (-.1,-.4) circle (.1);
		\draw[thick] (-1,0) -- (1,0);
		\draw[thick] (-.1,0) -- (.1,1);
		\draw[thick] (.1,-1) -- (.3,0);
		\draw[thick] (.1,0) -- (.9,1);
		\draw[thick] (.9,-1) -- (1,-.875);
		\draw[thick] (-1,-.875) -- (-.3,0);
		\draw[dotted] plot [smooth, tension=1] coordinates {(-1,.1) (-.2,0) (-1,-.1)};
		\draw[dotted] plot [smooth, tension=1] coordinates {(1,.1) (.2,0) (1,-.1)};
}}
\def\bigonhexpuncturedvi{\tikz[baseline=0]{
		\draw[->>] (-1,-1) -- (-1,.1);
		\draw (-1,.1) -- (-1,1);
		\draw[->>] (1,-1) -- (1,.1);
		\draw (1,.1) -- (1,1);
		\draw[->] (-1,1) -- (.1,1);
		\draw (.1,1) -- (1,1);
		\draw[->] (-1,-1) -- (.1,-1);
		\draw (.1,-1) -- (1,-1);
		\draw (0,-.4) circle (.1);
		\draw[thick] (-1,.1) -- (1,.1);
		\draw[thick] (-.5,.1) -- (-.5,-.1);
		\draw[thick] (-.3,.1) -- (.6,1);
		\draw[thick] (.6,-1) -- (1,-.6);
		\draw[thick] (-1,-.6) -- (-.5,-.1);
		\draw[thick] (.3,.1) -- (.9,1);
		\draw[thick] (.9,-1) -- (1,-.85);
		\draw[thick] (-1,-.85) -- (-.5,-.1);
}}
\def\bigonhexpuncturedvii{\tikz[baseline=0]{
		\draw[->>] (-1,-1) -- (-1,.1);
		\draw (-1,.1) -- (-1,1);
		\draw[->>] (1,-1) -- (1,.1);
		\draw (1,.1) -- (1,1);
		\draw[->] (-1,1) -- (.1,1);
		\draw (.1,1) -- (1,1);
		\draw[->] (-1,-1) -- (.1,-1);
		\draw (.1,-1) -- (1,-1);
		\draw (0,-.4) circle (.1);
		\draw[thick] (-1,0) -- (1,0);
		\draw[thick] (0,0) -- (.8,1);
		\draw[thick] (.8,-1) -- (1,-.75);
		\draw[thick] (-1,-.75) -- (-.4,0);
}}

\def\pgi{\tikz[baseline=0]{
		\draw[dashed] (0,0) circle (1);
		\draw[thick] (-.3,-.95) -- (0,-.65);
		\draw[thick] (.3,-.95) -- (0,-.65);
		\node at (0,-1.25) {$G_{1}$};
        \draw[fill] (0,-.4) circle (.02);
        \draw[fill] (0,-.65) circle (.02);
}}
\def\pgii{\tikz[baseline=0]{
		\draw[dashed] (0,0) circle (1);
		\draw[thick] (-.3,-.95) -- (0,-.65);
		\draw[thick] (.3,-.95) -- (0,-.65);
		\draw[thick] (0,0) circle (.4);
		\node at (0,-1.25) {$G_{2}$};
        \draw[fill] (0,-.4) circle (.02);
        \draw[fill] (0,-.65) circle (.02);
}}
\def\pgiii{\tikz[baseline=0]{
		\draw[dashed] (0,0) circle (1);
		\draw[thick] (-.3,-.95) -- (0,-.65);
		\draw[thick] (.3,-.95) -- (0,-.65);
		\draw[thick] (0,-.65) -- (0,-.4);
		\node at (0,-1.25) {$G_{3}$};
        \draw[fill] (0,-.4) circle (.02);
        \draw[fill] (0,-.65) circle (.02);
}}
\def\pgiv{\tikz[baseline=0]{
		\draw[dashed] (0,0) circle (1);
		\draw[thick] (-.3,-.95) -- (0,-.65);
		\draw[thick] (.3,-.95) -- (0,-.65);
		\draw[thick] (0,-.65) -- (0,-.4);
		\draw[thick] (0,0) circle (.4);
		\node at (0,-1.25) {$G_{4}$};
        \draw[fill] (0,-.4) circle (.02);
        \draw[fill] (0,-.65) circle (.02);
}}
\def\woikoi{\tikz[baseline=0]{
		\draw[->>] (-1,-1) -- (-1,.1);
		\draw (-1,.1) -- (-1,1);
		\draw[->>] (1,-1) -- (1,.1);
		\draw (1,.1) -- (1,1);
		\draw[->] (-1,1) -- (.1,1);
		\draw (.1,1) -- (1,1);
		\draw[->] (-1,-1) -- (.1,-1);
		\draw (.1,-1) -- (1,-1);
		\draw[thick] (-1,0) -- (1,0);
		\draw[thick] (-.2,0) -- (0,1);
		\draw[thick] (0,-1) -- (.2,0);
		\draw[thick] (-1,.5) -- (-.2,.5);
		\draw[thick] (0,.5) -- (1,.5);
}}
\def\woikoione{\tikz[baseline=0]{
		\draw[->>] (-1,-1) -- (-1,.1);
		\draw (-1,.1) -- (-1,1);
		\draw[->>] (1,-1) -- (1,.1);
		\draw (1,.1) -- (1,1);
		\draw[->] (-1,1) -- (.1,1);
		\draw (.1,1) -- (1,1);
		\draw[->] (-1,-1) -- (.1,-1);
		\draw (.1,-1) -- (1,-1);
		\draw[thick] (-1,0) -- (1,0);
		\draw[thick] (-.2,0) -- (-.2,.4);
		\draw[thick] (-.2,.4) arc[start angle = 0, end angle = 90, radius = .1];
		\draw[thick] (-1,.5) -- (-.3,.5);
		\draw[thick] (.2,0) -- (.2,-1);
		\draw[thick] (.2,1) -- (.2,.6);
		\draw[thick] (.2,.6) arc[start angle = 180, end angle = 270, radius = .1];
		\draw[thick] (.3,.5) -- (1,.5);
}}
\def\woikoitwo{\tikz[baseline=0]{
		\draw[->>] (-1,-1) -- (-1,.1);
		\draw (-1,.1) -- (-1,1);
		\draw[->>] (1,-1) -- (1,.1);
		\draw (1,.1) -- (1,1);
		\draw[->] (-1,1) -- (.1,1);
		\draw (.1,1) -- (1,1);
		\draw[->] (-1,-1) -- (.1,-1);
		\draw (.1,-1) -- (1,-1);
		\draw[thick] (-1,0) -- (1,0);
		\draw[thick] (.2,0) -- (.2,-1);
		\draw[thick] (-.2,0) -- (-.2,.2);
		\draw[thick] (.2,1) -- (.2,.8);
		\draw[thick] (-.2,.2) arc[start angle = 180, end angle = 100, radius = .1];
		\draw[thick] (-.15,.29) -- (1,.5);
		\draw[thick] (-1,.5) -- (.15,.71);
		\draw[thick] (.2,.8) arc[start angle = 0, end angle = -80, radius = .1];
}}
\def\woikoithree{\tikz[baseline=0]{
		\draw[->>] (-1,-1) -- (-1,.1);
		\draw (-1,.1) -- (-1,1);
		\draw[->>] (1,-1) -- (1,.1);
		\draw (1,.1) -- (1,1);
		\draw[->] (-1,1) -- (.1,1);
		\draw (.1,1) -- (1,1);
		\draw[->] (-1,-1) -- (.1,-1);
		\draw (.1,-1) -- (1,-1);
		\draw[thick] (-1,.2) -- (1,.2);
		\draw[thick] (-1,.-.2) -- (1,-.2);
		\draw[thick] (0,-.2) -- (0,.2);
		\draw[thick] (-.2,.2) -- (0,1);
		\draw[thick] (0,-1) -- (.2,-.2);
}}
\def\woikoifour{\tikz[baseline=0]{
		\draw[->>] (-1,-1) -- (-1,.1);
		\draw (-1,.1) -- (-1,1);
		\draw[->>] (1,-1) -- (1,.1);
		\draw (1,.1) -- (1,1);
		\draw[->] (-1,1) -- (.1,1);
		\draw (.1,1) -- (1,1);
		\draw[->] (-1,-1) -- (.1,-1);
		\draw (.1,-1) -- (1,-1);
		\draw[thick] (-1,.2) -- (-.4,.2);
		\draw[thick] (1,.2) -- (.4,.2);
		\draw[thick] (-1,-.2) -- (-.4,-.2);
		\draw[thick] (1,-.2) -- (.4,-.2);
		\draw[thick] (-.2,0) -- (-.4,.2);
		\draw[thick] (-.2,0) -- (-.4,-.2);
		\draw[thick] (.2,0) -- (.4,.2);
		\draw[thick] (.2,0) -- (.4,-.2);
		\draw[thick] (-.2,0) -- (.2,0);
		\draw[thick] (-.4,.2) -- (0,1);
		\draw[thick] (0,-1) -- (.4,-.2);
}}
\def\woikoixone{\tikz[baseline=0]{
		\draw[->>] (-1,-1) -- (-1,.1);
		\draw (-1,.1) -- (-1,1);
		\draw[->>] (1,-1) -- (1,.1);
		\draw (1,.1) -- (1,1);
		\draw[->] (-1,1) -- (.1,1);
		\draw (.1,1) -- (1,1);
		\draw[->] (-1,-1) -- (.1,-1);
		\draw (.1,-1) -- (1,-1);
		\draw[thick] (-1,.2) -- (1,.2);
		\draw[thick] (-1,.-.2) -- (1,-.2);
		\draw[thick] (-.2,.2) -- (0,1);
		\draw[thick] (0,-1) -- (.2,-.2);
}}
\def\woikoixtwo{\tikz[baseline=0]{
		\draw[->>] (-1,-1) -- (-1,.1);
		\draw (-1,.1) -- (-1,1);
		\draw[->>] (1,-1) -- (1,.1);
		\draw (1,.1) -- (1,1);
		\draw[->] (-1,1) -- (.1,1);
		\draw (.1,1) -- (1,1);
		\draw[->] (-1,-1) -- (.1,-1);
		\draw (.1,-1) -- (1,-1);
		\draw[thick] (-.4,.2) arc[start angle = 90, end angle = -90, radius = .2];
		\draw[thick] (.4,-.2) arc[start angle = 270, end angle = 90, radius = .2];
		\draw[thick] (-.4,.2) -- (-1,.2);
		\draw[thick] (-.4,-.2) -- (-1,-.2);
		\draw[thick] (.4,.2) -- (1,.2);
		\draw[thick] (.4,-.2) -- (1,-.2);
		\draw[thick] (-.4,.2) -- (0,1);
		\draw[thick] (.4,-.2) -- (0,-1);
}}
\def\woikoifive{\tikz[baseline=0]{
		\draw[->>] (-1,-1) -- (-1,.1);
		\draw (-1,.1) -- (-1,1);
		\draw[->>] (1,-1) -- (1,.1);
		\draw (1,.1) -- (1,1);
		\draw[->] (-1,1) -- (.1,1);
		\draw (.1,1) -- (1,1);
		\draw[->] (-1,-1) -- (.1,-1);
		\draw (.1,-1) -- (1,-1);
		\draw[thick] (-1,.2) -- (-.4,.2);
		\draw[thick] (1,.2) -- (.4,.2);
		\draw[thick] (-1,-.2) -- (-.4,-.2);
		\draw[thick] (1,-.2) -- (.4,-.2);
		\draw[thick] (-.2,0) -- (-.4,.2);
		\draw[thick] (-.2,0) -- (-.4,-.2);
		\draw[thick] (.2,0) -- (.4,.2);
		\draw[thick] (.2,0) -- (.4,-.2);
		\draw[thick] (-.2,0) -- (.2,0);
		\draw[thick] (0,0) -- (.2,1);
		\draw[thick] (.2,-1) -- (.4,-.2);
}}
\def\woikoixthree{\tikz[baseline=0]{
		\draw[->>] (-1,-1) -- (-1,.1);
		\draw (-1,.1) -- (-1,1);
		\draw[->>] (1,-1) -- (1,.1);
		\draw (1,.1) -- (1,1);
		\draw[->] (-1,1) -- (.1,1);
		\draw (.1,1) -- (1,1);
		\draw[->] (-1,-1) -- (.1,-1);
		\draw (.1,-1) -- (1,-1);
		\draw[thick] (0,-1) -- (0,-.2);
		\draw[thick] (0,0) circle (.2);
		\draw[thick] (0,1) -- (0,.2);
}}
\def\woiiioi{\tikz[baseline=0]{
		\draw[->>] (-1,-1) -- (-1,.1);
		\draw (-1,.1) -- (-1,1);
		\draw[->>] (1,-1) -- (1,.1);
		\draw (1,.1) -- (1,1);
		\draw[->] (-1,1) -- (.1,1);
		\draw (.1,1) -- (1,1);
		\draw[->] (-1,-1) -- (.1,-1);
		\draw (.1,-1) -- (1,-1);
		\draw[thick] (-1,-.5) -- (.5,1);
		\draw[thick] (.5,-1) -- (1,-.5);
		\draw[thick] (-1,0) -- (-.6,0);
		\draw[thick] (-.4,0) -- (1,0);
}}
\def\woiio{\tikz[baseline=0]{
		\draw[->>] (-1,-1) -- (-1,.1);
		\draw (-1,.1) -- (-1,1);
		\draw[->>] (1,-1) -- (1,.1);
		\draw (1,.1) -- (1,1);
		\draw[->] (-1,1) -- (.1,1);
		\draw (.1,1) -- (1,1);
		\draw[->] (-1,-1) -- (.1,-1);
		\draw (.1,-1) -- (1,-1);
		\draw[thick] (0,1) -- (0,-1);
}}
\def\woitwoo{\tikz[baseline=0]{
		\draw[->>] (-1,-1) -- (-1,.1);
		\draw (-1,.1) -- (-1,1);
		\draw[->>] (1,-1) -- (1,.1);
		\draw (1,.1) -- (1,1);
		\draw[->] (-1,1) -- (.1,1);
		\draw (.1,1) -- (1,1);
		\draw[->] (-1,-1) -- (.1,-1);
		\draw (.1,-1) -- (1,-1);
		\draw[thick] (-1,-.5) -- (1,.5);
		\draw[thick] (-1,.5) -- (0,1);
		\draw[thick] (0,-1) -- (1,-.5);
}}
\def\woinii{\tikz[baseline=0]{
		\draw[->>] (-1,-1) -- (-1,.1);
		\draw (-1,.1) -- (-1,1);
		\draw[->>] (1,-1) -- (1,.1);
		\draw (1,.1) -- (1,1);
		\draw[->] (-1,1) -- (.1,1);
		\draw (.1,1) -- (1,1);
		\draw[->] (-1,-1) -- (.1,-1);
		\draw (.1,-1) -- (1,-1);
		\draw[thick] (-1,0) -- (0,-1);
		\draw[thick] (0,1) -- (1,0);
}}
\def\woiii{\tikz[baseline=0]{
		\draw[->>] (-1,-1) -- (-1,.1);
		\draw (-1,.1) -- (-1,1);
		\draw[->>] (1,-1) -- (1,.1);
		\draw (1,.1) -- (1,1);
		\draw[->] (-1,1) -- (.1,1);
		\draw (.1,1) -- (1,1);
		\draw[->] (-1,-1) -- (.1,-1);
		\draw (.1,-1) -- (1,-1);
		\draw[thick] (-1,0) -- (0,1);
		\draw[thick] (0,-1) -- (1,0);
}}
\def\woiiooi{\tikz[baseline=0]{
		\draw[->>] (-1,-1) -- (-1,.1);
		\draw (-1,.1) -- (-1,1);
		\draw[->>] (1,-1) -- (1,.1);
		\draw (1,.1) -- (1,1);
		\draw[->] (-1,1) -- (.1,1);
		\draw (.1,1) -- (1,1);
		\draw[->] (-1,-1) -- (.1,-1);
		\draw (.1,-1) -- (1,-1);
		\draw[thick] (-1,0) -- (-.1,0);
		\draw[thick] (1,0) -- (.1,0);
		\draw[thick] (0,-1) -- (0,1);
}}

\author[Ellison]{Cole Ellison}
\address{Department of Mathematics, University of California, Santa Barbara, CA 93106, USA}
\email{coleellison@ucsb.edu}

\author[Higgins]{Vijay Higgins}
\address{Department of Mathematics, University of California, Los Angeles, CA 90095, USA}
\email{higginsv@math.ucla.edu}

\author[Sandey]{Mark Sandey}
\address{Department of Mathematics, University of California, Riverside, CA 92521, USA}
\email{mark.sandey@email.ucr.edu}

\title{$\rm{PGL}_2$ skein algebras of small surfaces}
\author[Sylvester]{Eden Sylvester}
\address{Department of Mathematics, University of California, Santa Barbara, CA 93106, USA}
\email{eeesylvester@gmail.com}

\thanks{
	2020 {\em Mathematics Subject Classification:} Primary 57K31. Secondary 17B37.\\
	{\em Key words and phrases: Skein algebra of a surface, product-to-sum formula, Jones-Wenzl projectors, chromatic skein algebra, Yamada polynomial}}

\begin{document}

\begin{abstract}
We study the $\rm{PGL}_2$ skein algebras of small surfaces, focusing on the closed torus and the once-punctured torus. We provide explicit finite presentations by generators and relations for the $\rm{PGL}_2$ skein algebra of the torus and once-punctured torus, analogous to those of Bullock and Przytycki for the Kauffman bracket skein algebras. We give explicit bases of the $\rm{PGL}_2$ skein algebras of the torus and once-punctured torus and describe the multiplicative structure of the basis for the torus, analogous to the product-to-sum formula of Frohman and Gelca. We twist the basis for the torus to obtain one with structure constants in $\mathbb{N}[q,q^{-1}].$ We compute the kernels of natural maps between the $\rm{PGL}_2$ and Kauffman bracket skein algebras associated to the once-punctured torus and closed torus. We compute the center of the $\rm{PGL}_2$ skein algebra of the torus, which depends on the specialization of the quantum parameter $q.$
\end{abstract}

\maketitle
\tableofcontents

\section{Introduction}

The Tutte polynomial, chromatic polynomial, and Yamada polynomials of trivalent graphs can be described using skein relations that are related to the Kauffman bracket skein relations. The relationship has been useful in studying TQFTs, planarity of graphs, graphs on tori, and symmetries of spatial graphs in works of Agol and Krushkal, Fendley and Krushkal, and Chbili \cite{AK16,FK09,FK10,FK23,Chb12}. The particular skein relations involved are modeled on the subcategory of the Temperley--Lieb category generated by the Jones--Wenzl projector on $2$ strands, and are most often referred to in the literature as either the $\rm{PGL}_2$ skein relations or $\rm{SO}_3$ skein relations. 

The goal of our paper is to find descriptions of the $\rm{PGL}_2$  skein algebras (defined in Section \ref{section PGL2 skein algebra}) of small surfaces analogous to those for the Kauffman bracket found by Frohman-Gelca \cite{FG00} and Bullock-Przytycki \cite{BP00}.

The relationship between $\rm{PGL}_2$ and Kauffman bracket skein modules of 3-manifolds is one focus of a more general Langlands duality conjecture for skein modules \cite{Jor23}, and we believe skein algebras of small surfaces should aid study in this direction. In the study of Kauffman bracket skein modules of 3-manifolds undertaken by Detcherry, Kalfagianni, and Sikora \cite{DKS24}, explicit presentations of the Kauffman bracket skein algebras of small surfaces played a prominent role. In particular, it was a key fact that the Kauffman bracket skein algebras of the 3-punctured sphere and the closed torus have simple explicit descriptions. We hope the analogous algebra presentations for the $\rm{PGL}_2$ skein algebras we present here will be similarly helpful in the study of $\rm{PGL}_2$ skein modules of 3-manifolds.

\subsection{Main results}

Denote by $\Sigma_{g,n}$ the surface obtained by removing $n$ points from a closed oriented genus $g$ surface. Let $\Sq(\Sigma_{g,n})$ and $\Y(\Sigma_{g,n})$ denote the Kauffman bracket and $\rm{PGL}_2$ skein algebras of $\Sigma_{g,n},$ respectively. Throughout the paper, we assume that $q+q^{-1} \neq 0.$ In Section \ref{preliminaries} we recall the relevant definitions and background on Kauffman bracket and $\rm{PGL}_2$ skein algebras. In Section \ref{torus section} we begin with the study of the $\rm{PGL}_2$ skein algebra of the torus, $\Y(\Sigma_{1,0}).$ We first obtain a basis for $\Y(\Sigma_{1,0}),$ which is analogous to the multicurve basis for the Kauffman bracket case except it contains one extra element, which is a web.

\begin{theorem}\label{intro basis torus}(Theorem \ref{Basis torus})
Let $(k,l) \in \Y(\Sigma_{1,0})$ denote the torus curve of slope $k/l$ and let $K \in \Y(\Sigma_{1,0})$ denote the element depicted by the following diagram on $\Sigma_{1,0}$: 
	\[K = \hexagon.\]
	Then
	\[\beta = \{\varnothing, K\} \cup \{(k,l)^{n} : \mathrm{gcd}(k,l) = 1, n \in \mathbb{Z}_{> 0}\}\]
	is a basis of $\Y(\Sigma_{1,0}).$
\end{theorem}

We then study the multiplicative structure associated with our basis, giving a product-to-sum formula for the $\rm{PGL}_2$ case which is inspired by the one of Frohman-Gelca \cite{FG00} for the Kauffman bracket case. We first construct a modification of our basis which is inspired by the Chebyshev basis used by Frohman-Gelca. Let $w \in \Y(\Sigma_{1,0})$ be given by $w=[2]K-(0,1)-(1,0)-(1,1)+[3]\varnothing.$ Let $g=\mathrm{gcd}(k,l)$ and let $(k,l)_Y=Y_g((\frac{k}{g},\frac{l}{g}))$ for the polynomial $Y_g$ defined later, in Equation (\ref{Y}). Then it follows from Theorem \ref{intro basis torus} that the set $\{(k,l)_{Y} : k,l \in \mathbb{Z}\} \cup \{w\}$ is a basis for $\Y(\Sigma_{1,0})$. The following theorem provides a complete multiplication table for this basis.

\begin{theorem}(Theorem \ref{torus presentation})
For any $k,l,r,s \in \mathbb{Z}$, the product of basis elements of $\Y(\Sigma_{1,0})$ is determined by the following relations:
	\[\begin{split}
	(k,l)_{Y} \cdot (r,s)_{Y} & = q^{2d}(k + r, l + s)_{Y} + q^{-2d}(k - r, l - s)_{Y} + \alpha w,\\
	(k,l)_{Y} \cdot w & = w \cdot (k,l)_{Y} = (-1)^{g_{1}}2w,\\
	w^{2} & = 2[2]^{2}w,\\
	\end{split}\]
    for scalars $\alpha,d,g_1$ depending on $k,l,r,s$ given in Theorem \ref{torus presentation}.
\end{theorem}

These theorems provide the following corollaries.

\begin{corollary}(Corollary \ref{positive torus basis})
Using the substitutions $e=\frac{1}{2[2]^2}w$ and $c_{r,s}=(1-e)(r,s)_Y$ we obtain a basis with structure constants in $\mathbb{N}[q,q^{-1}]$ for the algebra $\Y(\Sigma_{1,0})$ given by $P=\{e\}\cup\{c_{r,s} \mid r,s \in \mathbb{Z}\}.$ The product of pairs of elements of this basis is given by:
\[\begin{split}
	c_{k,l} \cdot c_{r,s} & = q^{2(rl-ks)}c_{k+r,s+l} + q^{-2(rl-ks)}c_{k-r,s-l},\\
	c_{k,l} \cdot e & =0= e \cdot c_{k,l},\\
	e^2 & = e.
\end{split}\]    
\end{corollary}

\begin{corollary}\label{intro Corollary zero divisors}(Corollaries \ref{JW not injective} and \ref{zero divisor})
The map $JW:\Y(\Sigma_{1,0}) \rightarrow \Sq(\Sigma_{1,0})$ from the $\rm{PGL}_2$ skein algebra of the torus to the Kauffman bracket skein algebra of the torus is not injective. Moreover, $\Y(\Sigma_{1,0})$ is not a domain and so there does not exist an injective algebra homomorphism $\Y(\Sigma_{1,0}) \rightarrow \Sq(\Sigma_{1,0}).$
\end{corollary}

We note that the fact from Corollary \ref{intro Corollary zero divisors} that $\Y(\Sigma_{1,0})$ is not a domain is in contrast with the case of the Kauffman bracket skein algebra $\Sq(\Sigma_{1,0})$, which is a domain by \cite{FG00} (and more generally, $\Sq(\Sigma_{g,n})$ is a domain for any oriented surface $\Sigma_{g,n}$ by \cite{BW11,PS19}). Additionally, the non-injectivity of $JW:\Y(\Sigma_{1,0}) \rightarrow \Sq(\Sigma_{1,0})$ from Corollary \ref{intro Corollary zero divisors} is in contrast to the injectivity of the analogous functor from the trivalent analogue of the Temperley-Lieb category to the Temperley-Lieb category, proven in \cite{FK10}.

We then study the center of $\Y(\Sigma_{1,0})$ for both $q$ generic and $q$ specialized to a root of unity.

\begin{theorem}(Theorem \ref{torus center})
\phantom\\
\begin{enumerate}[(a)]
\begin{item}
If $q$ is not a root of unity, then the center of $\Y(\Sigma_{1,0})$ is as follows:
\begin{equation*}
Z(\Y(\Sigma_{1,0})) = \mathrm{span}\{\varnothing,w\}.
\end{equation*}
\end{item}
\begin{item}
If $q=\omega_n$ is a primitive $n^{\text{th}}$ root of unity, then the center of $S_{\omega_{n}}^{Y}(\Sigma_{1,0})$ is
\begin{equation*}
Z(S_{\omega_{n}}^{Y}(\Sigma_{1,0})) = \mathrm{span}(\{(k,l)_{Y} : n \mid 4\mathrm{gcd}(k,l)\} \cup \{w\}).
\end{equation*}
\end{item}
\end{enumerate}
\end{theorem}

We note that in contrast to the Kauffman bracket case, the center of $\Y(\Sigma_{1,0})$ is not spanned by the empty diagram when $q$ is generic. Also, the image of the Frobenius homomorphism constructed in \cite{HSSSSW26} is always contained in the center of $\Y(\Sigma_{1,0})$ but the central element $w$ might not be in this image. See Remark \ref{center comparison} for a detailed comparison with the Kauffman bracket case.

In Section \ref{1-punctured torus section} we study the 1-punctured torus $\Sigma_{1,1}$ and use it to further study the closed torus.

\begin{theorem}(Theorem \ref{punctured torus presentation})
If $q^4 \neq 1$, then $\Y(\Sigma_{1,1})$ has a finite algebra presentation given by 4 generators modulo the relations appearing in Theorem \ref{punctured torus presentation}.
\end{theorem}

In Theorem \ref{puncture filling} we compute the kernel of the map $i_*: \Y(\Sigma_{1,1}) \rightarrow \Y(\Sigma_{1,0})$ induced by filling in the puncture of the torus. Our computation provides as a corollary a finite presentation of the $\rm{PGL}_2$ skein algebra of the closed torus.

\begin{corollary}(Corollary \ref{finite presentation torus})
If $q^4 \neq 1,$ $\Y(\Sigma_{1,0})$ has a finite presentation given by the 4 generators of $\Y(\Sigma_{1,1})$ modulo the relations from Theorem \ref{punctured torus presentation} along with the generators of $\ker(i_*)$ from Theorem \ref{puncture filling}.
\end{corollary}

\subsection{Acknowledgments}

This work was supported by National Science Foundation Grant No. DMS-2136090 and was completed during the Geometry \& Topology REU at UCLA during Summer 2025. We thank Ko Honda for overseeing the program, and we thank Roman Krutowski, Jung Joo Suh, and Yan Tao for their mentorship and support.  We thank David Jordan, Sam Gunningham, and Slava Krushkal for helpful discussions.

\section{Preliminaries}\label{preliminaries}

Let $\Sigma$ be an oriented surface. Let $q \in \mathbb{C}\setminus \{0\}$ denote a nonzero complex number such that $q+q^{-1} \neq 0$. The $n^{th}$ \textit{quantum integer} $[n]$ is defined by the formula
\[[n] = \sum_{i=1}^{n} q^{n
+1-2i}, n \in \mathbb{Z}_{> 0}.\]

\subsection{Kauffman bracket skein algebra}
We first recall the Kauffman bracket skein algebra of an oriented surface $\Sigma$.
The \textit{Kauffman bracket skein algebra}  $\Sq(\Sigma)$ of $\Sigma$ over $\mathbb{C}$ is the $\mathbb{C}$-vector space spanned by unoriented framed link diagrams on $\Sigma$ up to surface isotopy, modulo the Kauffman bracket skein relations:
\[\underset{\text{Kauffman bracket relation}}{\positivecrossing = q^{\frac{1}{2}} \verticalresolution + q^{-\frac{1}{2}} \horizontalresolution}, \quad \underset{\text{loop relation}}{\trivialcomponent = -[2]\varnothing}.\]
Multiplication of two skeins is defined on link diagrams $L_{1},L_{2} \in \Sq(\Sigma)$ by taking $L_{1} \cdot L_{2}$ to be the diagram given by placing $L_{1}$ over $L_{2}$, and multiplication on $\Sq(\Sigma)$ is given by extending this definition linearly.\\

\subsection{$\rm{PGL}_2$ skein algebra}\label{section PGL2 skein algebra} Next, we define the $\rm{PGL}_2$ skein algebra, which is the subject of this paper.
The \textit{$\rm{PGL}_2$ skein algebra} $\Y(\Sigma)$ of $\Sigma$ over $\mathbb{C}$ is the $\mathbb{C}$-vector space spanned by diagrams of knotted trivalent ribbon graphs (called \textit{webs}) on $\Sigma$ up to surface isotopy, modulo the following relations. We also allow vertexless components.
\begin{equation*}
\begin{aligned}
	& \underset{\text{crossing relation}}{\positivecrossing = q^{2} \verticalresolution + (q^{-2} - 1) \horizontalresolution + [2]\itangle},\\
	& \underset{\text{I-H relation}}{\itangle - \frac{1}{[2]} \horizontalresolution = \htangle - \frac{1}{[2]} \verticalresolution},\\
	& \underset{\text{monogon relation}}{\monogon = 0}, \quad \underset{\text{loop relation}}{\trivialcomponent = [3]\varnothing}.\\
\end{aligned}
\end{equation*}
Multiplication of two webs is defined on diagrams $w_{1},w_{2} \in \Y(\Sigma)$ by taking $w_{1} \cdot w_{2}$ to be the diagram given by placing $w_{1}$ over $w_{2}$, and extending this definition linearly. The defining skein relations imply the following useful relations, which will be used frequently to shorten computation:

\begin{equation}\label{ngon formulae}
\begin{aligned}
	\bigon & = \left( \frac{1 - [3]}{[2]} \right) \strand,\\
	\trigon & = \left( \frac{2 - [3]}{[2]} \right) \vertex.\\
\end{aligned}
\end{equation}

\subsection{Relationship between $\rm{PGL}_2$ and Kauffman bracket skein algebras}

We record the following properties for later use.

\begin{proposition}\label{properties}
	\phantom\\
	\begin{enumerate}[(i)]
		\begin{item}
			There is a well-defined skein algebra homomorphism $JW : \Y(\Sigma) \to \Sq(\Sigma),$ given by performing the following operations on edges and at vertices:
			\[ \quad \strand \mapsto \jwstrand, \quad \vertex \mapsto \jwvertex,\]
			where the box denotes the element
			\[\jwstrand = \jwidentity + \frac{1}{[2]} \jwreturn,\]
			which is the Jones--Wenzl idempotent of the Temperley-Lieb algebra on $2$ strands.
		\end{item}
		\begin{item} \cite{Prz91}
			Let $\Sigma_{1}$ and $\Sigma_{2}$ be oriented surfaces. Then an embedding $f : \Sigma_{1} \times I \to \Sigma_{2} \times I$ induces homomorphisms of the underlying module structure of the skein algebras
			\[f_{\ast} : \Sq(\Sigma_{1}) \to \Sq(\Sigma_{2}), f_{\ast}^{Y} : \Y(\Sigma_{1}) \to \Y(\Sigma_{2})\]
			given by applying the embedding to the associated link or web and extending linearly. An embedding $g : \Sigma_{1} \to \Sigma_{2}$ induces an algebra homomorphism by applying the embedding to the link or web diagram and extending linearly.
		\end{item}
		\begin{item}\cite{BW11,PS19}
			For any oriented surface $\Sigma$, $\Sq(\Sigma)$ has no zero divisors.
		\end{item}
		\begin{item}\cite{Prz91}
			A basis of $\Sq(\Sigma)$ is given by isotopy classes of crossingless diagrams on $\Sigma$ without loops which bound a disk.\\
		\end{item}
	\end{enumerate}
\end{proposition}

Part (i) is checked by a standard computation.

\subsection{The annulus, three-punctured sphere and once-punctured torus}

\begin{proposition}\label{JW injective on pants} \cite{Chb12,Chb06}
$JW: \mathcal{S}_q^Y(\Sigma) \rightarrow \mathcal{S}_q(\Sigma)$ is injective for both $\Sigma=\Sigma_{0,2}$ and $\Sigma=\Sigma_{0,3}$: the annulus and three-punctured sphere.  Additionally, $\mathcal{S}_q^Y(\Sigma_{0,2})\cong \mathbb{C}[a],$ where $a$ is the diagram represented by the core loop of the annulus. 
\end{proposition}

Although the skein relations Chbili uses involve graphs with no restrictions on the degrees of vertices, instead of only trivalent graphs, it is a routine computation to apply Chbili's methods to obtain the same results for our skein relations.

We next observe a useful corollary. Although we do not compute the $\rm{PGL}_2$ skein algebra of the once-punctured torus until Theorem \ref{punctured torus presentation}, we can use Proposition \ref{JW injective on pants} to observe that $JW$ is also injective on the once-punctured torus.

\begin{corollary}\label{JW injective on punctured torus}
The map $JW: \mathcal{S}_q^Y(\Sigma_{1,1}) \rightarrow \mathcal{S}_q(\Sigma_{1,1})$ is injective.    
\end{corollary}

\begin{proof}
The thickened surfaces $\Sigma_{1,1} \times I$ and $\Sigma_{0,3} \times I$ are diffeomorphic as 3-manifolds. Let $f : (\Sigma_{1,1}) \times I \to (\Sigma_{0,3}) \times I$ be such a diffeomorphism.

By Proposition \ref{JW injective on pants}, the map $JW: \Y(\Sigma_{0,3}) \rightarrow \Sq(\Sigma_{0,3})$ is injective. Recall the functoriality of $\Sq$ and $\Y$ described in part (ii) of Proposition \ref{properties}. This functoriality, along with the injectivity of $JW$ on $\Y(\Sigma_{0,3})$, implies that $JW$ is injective on $\Y(\Sigma_{1,1})$ by examining the following commutative diagram of skein modules:
	\[\begin{tikzcd}
		\Y(\Sigma_{1,1}) \arrow[d, "f_{\ast}^{Y}", "\sim"'] \arrow[r, hook, "JW"] & \Sq(\Sigma_{1,1})\\
		\Y(\Sigma_{0,3}) \arrow[r, hook, "JW"] & \Sq(\Sigma_{0,3}) \arrow[u, "f_{\ast}^{-1}", "\sim"'].\\
	\end{tikzcd}\]
\end{proof}

\section{The torus}\label{torus section}

Let $(k,l)$ denote the $(k,l)$ torus knot, which traverses the meridian $k$ times and the longitude $l$ times at a slope of $k/l$, as follows:
\[\pqknot.\]
Recall that we consider elements of the skein algebra as unoriented, so $(k,l) = (-k,-l)$. For the sake of computation, however, take the convention that $l \ge 0$ and $l = 0 \implies k > 0$. We will omit an indication of orientation of the knots except where relevant.

By part (iv) of Proposition \ref{properties}, a basis for the Kauffman bracket skein algebra $\mathcal{S}_q(\Sigma_{1,0})$ is given by the set of \textit{multicurves} on the surface: non-negative powers of $(k,l)$, for $k,l$ coprime.

\subsection{A basis for the torus}
In this section we will construct the following basis for the $\rm{PGL}_2$ skein algebra of the torus $\Y(\Sigma_{1,0})$ analogous to the multicurve basis of the Kauffman bracket skein algebra $\Sq(\Sigma_{1,0}).$

\begin{theorem}\label{Basis torus}
Let
	\[K = \hexagon.\]
	Then
	\[\beta = \{\varnothing, K\} \cup \{(k,l)^{n} : \mathrm{gcd}(k,l) = 1, n \in \mathbb{Z}_{> 0}\}\]
	is a basis of $\Y(\Sigma_{1,0})$.
\end{theorem}

\begin{proof}
The proof follows from the following Lemma \ref{span torus} and Corollary \ref{LI beta}.
\end{proof}

\begin{lemma}\label{span torus}
$\beta$ is a spanning set of $\Y(\Sigma_{1,0})$.
\end{lemma}

\begin{proof}
Any element in $\Y(\Sigma_{1,0})$ may be rewritten as a linear combination of diagrams without crossings. It suffices to show that a diagram that is a union of crossingless connected components is in the span of $\beta.$ Consider such a diagram $D$ in $\Y(\Sigma_{1,0}).$ For each connected component $G,$ the inclusion $i : G \to \Sigma_{1,0}$ induces a homomorphism of first homology groups $i_{\ast} : H_{1}(G; \mathbb{Z}) \to H_{1}(\Sigma_{1,0}; \mathbb{Z})$. We will denote the \textit{rank} of $G$ by the rank of the image of $i_{\ast}.$

If all components $G$ of $D$ are rank-zero graphs, then each component is contained in a disk and can be reduced to a scalar.

If one component $G$ of $D$ is a rank-one graph on $\Sigma_{1,0}$, then any other component $F$ of $D$ is contained in $\Sigma \setminus G$ and so $F$ is either a rank-one graph or rank-zero graph. In this case, $D$ is a union of graphs of rank zero or rank one. Each component is contained in an annulus or in a disk and is in the span of multicurves in $\Y(\Sigma_{1,0})$ by using the fact that $\Y(\Sigma_{0,2})=\mathbb{C}[a]$ and $\Y(\Sigma_{0,1})=\mathbb{C}$ and then applying part (ii) of Proposition \ref{properties}.

If one component $G$ of $D$ is a rank-two graph on $\Sigma_{1,0}$ then all other components are contained in $\Sigma \setminus G$ and must be rank-zero graphs. So any component other than $G$ can be reduced to a scalar. To finish the proof, we must show that a connected rank-two graph $G$ on $\Sigma_{1,0}$ can be rewritten as a linear combination of $\beta.$ To do this, we recall \cite{Dew73,Neg94} that two rank-two graphs with an equal number of vertices are related by I-H moves.

If $G$ is a connected rank-two graph on $\Sigma_{1,0}$ with exactly $2$ vertices, then since $K$ is also a rank-two graph on $\Sigma_{1,0}$ with 2 vertices,  we can apply I-H skein relations to $G$ in order to rewrite it as $G=K+G_0 \in \mathrm{span}(\beta)$ where $G_0$ is a linear combination of multicurves.

We then proceed by induction. If $G$ is a connected rank-two graph on $\Sigma_{1,0}$ with $2n$ vertices for $n>1$ then we can use I-H skein relations to rewrite $G$ as the following graph $K_{n - 1}$ plus terms with fewer vertices:

\[K_{n - 1} =\ranktwomonogon.\]

Since $K_{n - 1} = 0$ by the monogon relation, we have rewritten $G$ as a linear combination of terms with fewer vertices. By induction, we can rewrite $G$ as a linear combination of multicurves and $K.$ So $G \in \mathrm{span}(\beta).$

\end{proof}

Showing that $\beta$ is linearly independent requires a more subtle argument due to two main reasons. First, although we will use the map $JW: \Y(\Sigma_{1,0})\rightarrow \Sq(\Sigma_{1,0})$ to show that the multicurve elements of $\beta$ are linearly independent, it turns out that $JW(K)$ is in $\text{span}(JW(\beta\setminus\{K\}))$ so it is not immediately clear that $K$ is linearly independent from the multicurves. Second, the defining skein relations for $\rm{PGL}_2$ webs do not satisfy an easy notion of confluence in the sense of \cite{SW07} in the same way as the Kauffman bracket relations or even the Kuperberg webs \cite{Kup96}. Instead, we will make use of a specialization of Fendley and Krushkal's version of the flow polynomial for torus graphs from \cite{FK23}.

In Eq. (3.4) of \cite{FK23}, Fendley and Krushkal define $P_G(Y,W,A)$ which is a polynomial invariant for a torus graph $G.$ Here we will use a renormalized version of $P_G$ and use the specialization $Y=[2]^2, W=1$ to yield a polynomial invariant $\hat{P}_G(A)$ for elements $G$ of $\Y(\Sigma_{1,0})$ represented by a torus graph $G.$ We next describe the definition of $\hat{P}_G(A).$

Given a crossingless web on $\Sigma_{1,0}$ we view it as a graph $G$ where any loop components are considered to consist of 1 vertex and 1 edge. For any subgraph $H \subset G$ let $n(H)=\mathrm{rank}(H_1(H;\mathbb{Z})),$ which does not depend on the embedding of $H$, and let $s(H)=1$ if $H$ is a rank-two graph on $\Sigma_{1,0}$ and $s(H)=0$ if $H$ is a rank-zero or rank-one graph on $\Sigma_{1,0}$. Let $V(G)$ denote the number of trivalent vertices of $G$ and let $E(G)$ and $E(H)$ denote the numbers of edges of $G$ and $H$, respectively.  Then the invariant $\hat{P}_G\in \mathbb{C}[A]$ is defined as follows.

\begin{equation}\label{hat p}
\hat{P}_G(A)=\left(\frac{-1}{[2]} \right)^{E(G) - V(G)} \sum_{H \subset G} (-1)^{E(G)-E(H)} [2]^{2n(H)}A^{s(H)},
\end{equation}
where the sum is taken over all spanning subgraphs $H$ of $G.$

\begin{proposition}\label{invariant}
The assignment $G \mapsto \hat{P}_G(A)$ extends to a well-defined $\mathbb{C}\text{-linear}$ homomorphism $\Y(\Sigma_{1,0}) \rightarrow \mathbb{C}[A].$
\end{proposition}

\begin{proof}
The definition of the assignment $G \mapsto \hat{P}_G(A)$ for crossingless $G$ extends to diagrams with crossings by requiring that we expand each crossing using the crossing relation before applying $\hat{P}$. To check that the map is well-defined on the skein module, we only need to check that it respects the I-H relation, monogon relation, and loop relation. 

The polynomial $P_{G}(Y,W,A)$ from \cite{FK23} already satisfies the I-H relation, as they have verified. Consequently, our specialization $\hat{P}_G(A)=P_G([2]^2,1,A)$ satisfies the I-H relation as well. We will check that the monogon and loop skein relations are satisfied. Let $M(H)$ denote the monomial given by the summand corresponding to the subgraph $H$ in Equation (\ref{hat p}).\\

Suppose that $\tilde{G}$ is a crossingless graph on the torus and contains a monogon. Let $G$ denote the graph given by removing the monogon from $\tilde{G}$. If $H$ is a spanning subgraph of $G,$ let $M(H)$ denote the monomial given by the summand corresponding to the subgraph $H$ in Equation (\ref{hat p}).

In a neighborhood $U$ of the monogon of $\tilde{G}$, a spanning subgraph $\tilde{H}$ of $\tilde{G}$ looks like one of the following four diagrams.

\[\pgi, \quad \pgii, \quad \pgiii, \quad \pgiv.\]

Here, the edges at the bottom of the diagram are drawn only to represent placeholders for the rest of $\tilde{H}.$

We can partition $\hat{P}_{\tilde{G}}$ into sums of subgraphs based on how they look in the neighborhood $U$. We may compute

\begin{align*}
\hat{P}_{\tilde{G}}(A)&=\left( \frac{-1}{[2]} \right)^{E(\tilde{G})-V(\tilde{G})}\sum_{\tilde{H}\subset \tilde{G}}M(\tilde{H})\\
&= \left( \frac{-1}{[2]} \right)^{E(\tilde{G})-V(\tilde{G})}\sum_{H\subset G}\sum_{i=1}^4 M(\tilde{H}_i),\\
\end{align*}

where $\tilde{H}_i=H \cup G_i$ is the unique decomposition of a spanning subgraph of $\tilde{G}$ into the union of a spanning subgraph $H$ of $G$ and one of the graphs $G_i=\tilde{H} \cap U$ from above. Since $\tilde{H}_1$ and $\tilde{H}_3$ have the same ranks and nullities, but $E(\tilde{H}_3)=E(\tilde{H}_1)+1,$ we have that the contributions of $M(\tilde{H}_1)$ and $M(\tilde{H}_3)$ cancel. Similarly, since $\tilde{H}_4$ has one more edge than $\tilde{H}_2$ but the same rank and nullity, the contributions from these two graphs cancel as well. Thus, if $\tilde{G}$ contains a monogon, we have $\hat{P}_{\tilde{G}}(A)=0,$ as required.

Similarly, we may consider a graph $\tilde{G}$ with a trivial loop component, and a graph $G$ given by removing the trivial loop component. Let $\tilde{H} \subseteq \tilde{G}$ be a spanning subgraph. In a neighborhood $U$ of the trivial component of $\tilde{G}$, $\tilde{H} \cap U$ consists of either a single vertex or the trivial loop consisting of a single vertex and a single edge. We can partition $\hat{P}_{\tilde{G}}$ into sums of subgraphs based on how they look in the neighborhood $U$:
\begin{align*}
\hat{P}_{\tilde{G}}(A) & = \left( \frac{-1}{[2]} \right)^{E(\tilde{G}) - V(\tilde{G})} \sum_{\tilde{H} \subseteq \tilde{G}} M(\tilde{H})\\
&=\left( \frac{-1}{[2]} \right)^{E(\tilde{G})-V(\tilde{G})}\sum_{H\subset G}\sum_{i=1}^2 M(\tilde{H}_i),\\
\end{align*}

where $\tilde{H}_1$ satisfies $\tilde{H}_1\cap U$ consists of a single vertex and $\tilde{H}_{2}$ satisfies $\tilde{H}_2 \cap U$ consists of the trivial loop consisting of a single vertex and a single edge.

Since $\tilde{H}_{2}$ has the same rank as $\tilde{H}_{1}$, one more edge, and nullity one greater, we have the equality
\[M(\tilde{H}_{2}) = -[2]^{2}M(\tilde{H}_{1}).\]
Thus,
\begin{align*}
P_{\tilde{G}}(A) & = \left( - \frac{1}{[2]} \right)^{E(\tilde{G}) - V(\tilde{G})} \sum_{H \subset G} \left( 1 - [2]^{2} \right) M(\tilde{H}_{1})\\
& = \left( - \frac{1}{[2]} \right)^{E(G) - V(G)} \sum_{H \subset G} -\left( 1 - [2]^{2} \right) M(G)\\
& = ([2] - 1) \hat{P}_{G}(A)\\
& = [3] \hat{P}_{G}(A)\\
\end{align*}
Therefore, all planar skein relations are satisfied. It follows that $G \mapsto \hat{P}_{G}(A)$ is well-defined on $\Y(\Sigma_{1,0})$.
\end{proof}

\begin{corollary}\label{LI beta}
$\beta$ is a linearly independent subset of $\Y(\Sigma_{1,0}).$
\end{corollary}

\begin{proof}

First, we will use the invariant $\hat{P}_K(A)$ to observe that the element $K$ is not in the span of $\beta \setminus \{K\}$, which consists of the multicurve elements of $\beta$. We observe that since $K$ is a rank-two graph on $\Sigma_{1,0}$ containing no other rank-two subgraph, $\hat{P}_K(A)$ is a degree-one polynomial of $\mathbb{C}[A].$ On the other hand, any multicurve consists of only graphs of rank zero or rank one, and so for any multicurve $c,$ $\hat{P}_c(A) \in \mathbb{C}.$ Since by Proposition \ref{invariant}, $\hat{P}$ defines a linear map, we conclude that $K$ is not in the span of the multicurves.

To finish the proof, we verify that $\beta \setminus \{K\}$ is linearly independent by showing that $JW$ maps $\beta \setminus \{K\}$ injectively to a linearly independent set in $\Sq(\Sigma_{1,0}).$ We see that for each torus curve $(k,l) \in \Y(\Sigma_{1,0})$ for $k,l$ coprime we have by the definition of the $JW$ map, $JW((k,l))=(k,l)^2 - \varnothing \in \Sq(\Sigma_{1,0}).$

Thus,

\[JW((k,l)^n)=((k,l)^2-\varnothing)^n.\]

By using the fact that the set of multicurves is a linearly independent set in $\Sq(\Sigma_{1,0}),$ we conclude that $\beta \setminus \{K\}$ is linearly independent in $\Y(\Sigma_{1,0}).$ Now, since $K$ is not in the span of $\beta \setminus \{K\}$ we have that $\beta$ is linearly independent.
\end{proof}

We observe the following corollary, which is in contrast to the injectivity of $JW$ in the cases of $\Sigma=\Sigma_{0,1},\Sigma_{0,2},\Sigma_{0,3}$ from \cite{Chb12}.

\begin{corollary}\label{JW not injective}
The map $JW: \Y(\Sigma_{1,0}) \rightarrow \Sq(\Sigma_{1,0})$ is not injective for the torus. Its kernel is 1-dimensional.
\end{corollary}

\begin{proof}

Let

\begin{equation}\label{def of w}
w=[2]K-(0,1)-(1,0)-(1,1)+[3] \varnothing \in \Y(\Sigma_{1,0}).
\end{equation}

Since $[2] \neq 0$ and $K$ is not in the span of the multicurves we have that $w \neq 0.$ However, a routine computation yields that $JW(w)=0\in \Sq(\Sigma_{1,0})$ (see the proof of Lemma \ref{w^2} for a similar computation). Therefore, $JW: \Y(\Sigma_{1,0}) \rightarrow \Sq(\Sigma_{1,0})$ is not injective. Since $JW$ is injective when restricted to the multicurve elements of the basis $\beta,$ we have that the kernel of $JW$ is precisely $\text{span}(w).$
\end{proof}

\subsection{Kauffman bracket skein algebra of the torus}

We first recall the presentation of the Kauffman bracket skein algebra of the torus given by the product-to-sum formula of Frohman-Gelca \cite{FG00}. The Chebyshev polynomials of the first kind, $T_n(x)$, are given recursively by

\[T_{0}(x) = 2, T_{1}(x) = x, T_{n}(x) = xT_{n - 1}(x) - T_{n - 2}(x).\]

The multicurve basis for $\Sq(\Sigma_{1,0})$ is given by the torus links $\{(k,l)^{n} : \mathrm{gcd}(k,l) = 1, n \in \mathbb{Z}_{\geq 0}\}.$ There is a unitriangular change of basis to the Frohman-Gelca Chebyshev basis given by $(k,l)^n \mapsto T_n((k,l)).$ The multiplication of these basis elements is given by the following product-to-sum presentation.

We use the following notation.

Let $k,l \in \mathbb{Z}$. Let $g = \mathrm{gcd}(k,l)$, with the convention $\mathrm{gcd}(0,0) = 0$. Define
\[(k,l)_{T} = T_g \left( \frac{k}{g}, \frac{l}{g} \right)\]
for $g$ nonzero and $(0,0)_{T} = 2\varnothing$.
\begin{theorem}\cite{FG00}
	For any $k,l,r,s \in \mathbb{Z}$,
	\[(k,l)_{T} \cdot (r,s)_{T} = q^{-\frac{1}{2}\begin{vmatrix} k & l \\ r & s \end{vmatrix}}(k + r,l + s)_{T} + q^{\frac{1}{2}\begin{vmatrix} k & l \\ r & s \end{vmatrix}}(k - r,l - s)_{T}.\]
\end{theorem}

\subsection{$\rm{PGL}_2$ skein algebra of the torus and product-to-sum formula}

For the analogues of Chebyshev polynomials for $\rm{PGL}_2$ we will use polynomials $Y_n$ given by

\begin{equation}\label{Y}
Y_{0}(x) = 2, Y_{1}(x) = x - 1, Y_{2}(x) = x^{2} - 2x - 1, Y_{n}(x) = xY_{n - 1}(x) - xY_{n - 2}(x) + Y_{n - 3}(x).
\end{equation}

It can be checked by induction that the polynomial $Y_n$ satisfies

\[Y_{n}(x^{2} - 1) = T_{2n}(x).\]

Let $k,l \in \mathbb{Z}$. Let $g = \mathrm{gcd}(k,l)$, with the convention $\mathrm{gcd}(0,0) = 0$. Define

\[(k,l)_{Y} = Y_g \left( \frac{k}{g}, \frac{l}{g} \right)\]
for $g$ nonzero and $(0,0)_{Y} = 2\varnothing$.
Recall the element $w \in \Y(\Sigma_{1,0})$ from Equation (\ref{def of w}). Since each $Y_n$ is a monic polynomial of degree $n$ and since $[2] \neq0,$ it follows from Theorem \ref{Basis torus} that the set $\{(k,l)_{Y} : k,l \in \mathbb{Z}\}\cup\{w\}$ is a basis for $\Y(\Sigma_{1,0})$. We will show the product-to-sum formulas governing the multiplication in this basis are the following:

\begin{theorem}\label{torus presentation}
	For any $k,l,r,s \in \mathbb{Z}$,
	\[\begin{split}
	(k,l)_{Y} \cdot (r,s)_{Y} & = q^{2d}(k + r, l + s)_{Y} + q^{-2d}(k - r, l - s)_{Y} + \alpha w,\\
	(k,l)_{Y} \cdot w & = w \cdot (k,l)_{Y} = (-1)^{g_{1}}2w,\\
	w^{2} & = 2[2]^{2}w,\\
	\end{split}\]
	where
	\[\begin{split}
		d & = - \begin{vmatrix} k & l \\ r & s \end{vmatrix},\\
		\alpha & = \frac{(-1)^{g_{1} + g_{2}}2 - (-1)^{g_{3}}q^{2d} - (-1)^{g_{4}}q^{-2d}}{[2]^{2}},\\
		g_{1} & = \mathrm{gcd}(k,l),\\
		g_{2} & = \mathrm{gcd}(r,s),\\
		g_{3} & = \mathrm{gcd}(k + r, l + s),\\
		g_{4} & = \mathrm{gcd}(k - r, l - s).\\
	\end{split}\]
\end{theorem}

\begin{proof}

The proof follows from Lemma \ref{w^2}, Lemma \ref{w*(r,s)_Y}, and Lemma \ref{product of threaded curves}.
\end{proof}

\begin{lemma}\label{w^2}
The identity $w^2=2[2]^2w$ holds in $\Y(\Sigma_{1,0}).$
\end{lemma}

\begin{proof}
The fact that $w^2=2[2]^2w$ can be verified by direct computation but we find it easier to lift the computation to a corresponding computation on the once-punctured torus where we can use the injectivity of the map $JW: \Y(\Sigma_{1,1}) \rightarrow \Sq(\Sigma_{1,1})$ from Corollary \ref{JW injective on punctured torus}.

By abuse of notation, let $K, (k,l) \in \Y(\Sigma_{1,1})$ denote elements represented by the same underlying diagrams as the elements $K,(k,l) \in \Y(\Sigma_{1,0})$, but existing in the punctured torus. See Section \ref{1-punctured torus section} for details. The canonical embedding
\[i : \Sigma_{1,1} \to \Sigma_{1,0}\]
induces an epimorphism of skein algebras
\[i_{\ast}^{Y} : \Y(\Sigma_{1,1}) \to \Y(\Sigma_{1,0}).\]

We can compute $w^{2}$ in $\Y(\Sigma_{1,0})$ by squaring the image of its corresponding element in $\Y(\Sigma_{1,1})$ under $JW$, then finding the unique preimage. Define $x$ to be the boundary-parallel curve. We can compute that in $\Y(\Sigma_{1,1})$,
	\[\begin{split}\label{jwk}
		JW(K) & = \jwktilde\\
			      & = \jwktildezero + \frac{1}{[2]} \jwktildeone + \frac{1}{[2]} \jwktildetwo + \frac{1}{[2]^{2}} \jwktildethree\\
			      & \quad + \frac{1}{[2]} \jwktildefour \frac{1}{[2]^{2}} \jwktildefive + \frac{1}{[2]^{2}} \jwktildesix + \frac{1}{[2]^{3}} \jwktildeseven.\\
			      & = x + \frac{1}{[2]} (1,1)^{2} + \frac{1}{[2]} (0,1)^{2} + \frac{1}{[2]^{2}}(-[2]\varnothing)\\
			      & \quad + \frac{1}{[2]} (1,0)^{2} + \frac{1}{[2]^{2}} (-[2]\varnothing) + \frac{1}{[2]^{2}} (-[2]\varnothing) + \frac{1}{[2]^{3}}(-[2]\varnothing)^{2}\\
			      & = x + \frac{1}{[2]}((0,1)^{2} + (1,0)^{2}  + (1,1)^{2} - 2\varnothing),\\
	\end{split}\]
	so
	\[\begin{split}
		JW(w^{2}) & = (JW(w))^{2}\\
				  & = (JW([2]K - (0,1) - (1,0) - (1,1) + [3]\varnothing))^{2}\\
				  & = ([2]JW(K) - JW((0,1)) - JW((1,0)) - JW((1,1)) + [3]\varnothing)^{2}\\
				  & = ([2]x + (0,1)^{2} + (1,0)^{2} + (1,1)^{2} - 2\varnothing\\
				  & \quad - ((1,0)^{2} - \varnothing) - ((0,1)^{2} - \varnothing) - ((1,1)^{2} - \varnothing) + [3]\varnothing)^{2}\\
				  & = ([2]x + (1 + [3])\varnothing)^{2}\\
				  & = ([2]x + [2]^{2}\varnothing)^{2}\\
				  & = [2]^{2}x^{2} + 2[2]^{3}x + [2]^{4}\varnothing,\\
	\end{split}\]
	which has unique preimage in $\Y(\Sigma_{1,1})$ of
	\[\begin{split}
		w^{2} & = JW^{-1}(JW(w)^{2})\\
			      & = [2]^{2}(x + \varnothing) + 2[2]^{2}(w - [2]^{2} \varnothing) + [2]^{4} \varnothing.\\
	\end{split}\]
	Now, the epimorphism from the punctured torus to the unpunctured torus allows for the boundary-parallel curve to be resolved, mapping to $[3]\varnothing$. Thus, in $\Y(\Sigma_{1,0})$ we have
	\[\begin{split}
		w^{2} & = i_{\ast}^{Y}(w^{2})\\
		      & = [2]^{2}([3] + 1)\varnothing + 2[2]^{2}(w - [2]^{2} \varnothing) + [2]^{4}\varnothing\\
		      & = 2[2]^{2}w.\\
	\end{split}\]
\end{proof}

To make the other computations easier, we make the observation that $w=h_*(w)$ if $h$ is an orientation-preserving diffeomorphism of $\Sigma_{1,0}.$

\begin{lemma}\label{h_*(w)=w}
If $h: \Sigma_{1,0} \rightarrow \Sigma_{1,0}$ is an orientation-preserving diffeomorphism of $\Sigma_{1,0}$ then we have $h_*(w)=w.$
\end{lemma}
\begin{proof}
The map $JW$ commutes with $h_*$, so
\[JW(h_{\ast}(w)) = h_{\ast}(JW(w)) = h_{\ast}(0) = 0.\]
Thus, $h_{\ast}(w) \in \mathrm{ker}(JW)$. By Corollary \ref{JW not injective}, $h_*(w) \in \text{span}(w)$ and there exists $c \in \mathbb{C}$ such that
\begin{equation}\label{h(w) 1}
 h_*(w)=cw.   
\end{equation}
On the other hand, we can directly compute that

\begin{align*}
h_*(w)&=h_*\left([2]K-(1,0)-(0,1)-(1,1)+[3]\varnothing\right)\\
&= [2]h_*(K)-h_*((1,0))-h_*((0,1))-h_*((1,1))+[3]h_*(\varnothing).
\end{align*}

Since $h_*$ preserves the number of vertices and rank, we have that $h_*(K)$ is a two-vertex rank-two graph. Therefore, it is I-H equivalent to $K$. Thus, by applying I-H skein relations to $h_*(w)$ we have that

\begin{equation}\label{h(w) 2}
 h_*(w)=[2]K + (\text{linear combination of multicurves}).   
\end{equation}

Since we also have that

\begin{equation}\label{h(w) 3}
 w=[2]K+(\text{linear combination of multicurves}),   
\end{equation}

and since $K$ is a basis element not spanned by multicurves, equations (\ref{h(w) 1}), (\ref{h(w) 2}), and (\ref{h(w) 3}) imply that $c=1.$ We conclude that $h_*(w)=w.$
\end{proof}

\begin{lemma}\label{w*(r,s)}
The product of a torus knot and $w$ is given by
\begin{equation*}
(r,s) \cdot w=-w=w \cdot (r,s).
\end{equation*}
\end{lemma}

\begin{proof}
We wish to compute $w \cdot (r,s)$. Since $\mathrm{ker}(JW)$ is an ideal spanned by $w$, $w \cdot (r,s) = cw$ for some $c \in \mathbb{C}$. We claim $c$ does not depend on $(r,s)$.\\
Let $h : \Sigma_{1,0} \to \Sigma_{1,0}$ be an orientation-preserving diffeomorphism sending the curve of slope $(0,1)$ to the curve of slope $(r,s)$. Then we have that the induced map $h_*$ is an algebra homomorphism $h_*: \Y(\Sigma_{1,0})\rightarrow \Y(\Sigma_{1,0})$ and satisfies $h_*(w)=w$ by Lemma \ref{h_*(w)=w}. Consequently, we have
\begin{align*}
	h_{\ast}(cw) & = cw\\
		     & = w \cdot (r,s)\\
		     & = h_{\ast}(w) \cdot h_{\ast}((0,1))\\
		     & = h_{\ast}(w \cdot (0,1)).
\end{align*}
Since $h$ is a diffeomorphism, $h_{\ast}$ is actually an isomorphism, so $cw = w \cdot (0,1)$. Thus, $c$ does not depend on $(r,s)$, and we may compute the product of any curve with $w$ by diagrammatic computation of $w \cdot (0,1)$.

Since we know that $w \cdot (0,1)=cw=c([2]K-(1,0)-(0,1)-(1,1)+[3]\varnothing)$, we can compute the value of $c$ by computing the coefficient of $K$ when expressing $w \cdot(0,1)$ as a linear combination of $K$ and multicurves. The computation of the coefficient of $K$ can be achieved by summing the coefficients associated to diagrams that are rank-two webs with two vertices since these are all I-H equivalent to $K.$ We can compute
\begin{align*}
	K \cdot (0,1) & = \woikoi\\
		      & = q^{2} \woikoione + (q^{-2} - 1) \woikoitwo + [2] \woikoithree.
\end{align*}
The first two terms are rank-two two-vertex webs and thus contribute $q^{2}$ and $q^{-2} - 1$ to the $K$ coefficient of the product. The third term may be reduced as
\[[2] \woikoithree = [2] \woikoifour + \woikoixone - \woikoixtwo.\]
The second two terms on the right-hand side do not contribute to the coefficient of $K$ as they are not rank-two. The first term may be reduced as
\[[2] \woikoifour = [2] \woikoifive + \woikoitwo - \woikoixthree.\]
The second term of the right side contributes $1$ to the coefficient of $K$, the third term does not contribute as it is not rank-two, and the first term contributes $2-[3]$ after applying the following triangle reduction from Equation (\ref{ngon formulae}):
\[[2] \woikoifive = (2 - [3]) \hexagon.\]
Next, we can compute
\begin{align*}
	(1,0) \cdot (0,1) & = \woiiooi\\
			  & = q^{2} \woinii + (q^{-2} - 1) \woiii + [2] \hexagon,\\
	(1,1) \cdot (0,1) & = \woiiioi\\
			  & = q^{2} \woiio + (q^{-2} - 1) \woitwoo + [2] \woikoitwo,
\end{align*}
each of which only contributes to the coefficient of $K$ by the $[2]$ in the third term. The products $(0,1) \cdot (0,1)$ and $\varnothing \cdot (0,1)$ do not contribute. Summing the coefficients, we have that the coefficient of $K$ in $w \cdot (0,1)$ is
\[[2](q^{2} + q^{-2} - 1 + 1 + 2 - [3]) - [2] - [2] = -[2].\]
Then $c = -[2]/[2] = -1$. Therefore, $w \cdot (0,1) = -w$, so $w \cdot (r,s) = -w$ for all curves $(r,s)$. The computation of $(r,s) \cdot w=-w$ follows from symmetry by replacing $q$ with $q^{-1}.$
\end{proof}

Now recall the definitions of the polynomials $Y_n$ from Equation (\ref{Y}) and the elements $(r,s)_Y.$

\begin{lemma}\label{w*(r,s)_Y}
\begin{equation*}
(r,s)_{Y} \cdot w = (-1)^{\mathrm{gcd}(r,s)}2w= w \cdot (r,s)_{Y}.
\end{equation*}
\end{lemma}

\begin{proof}
We will prove this by induction. We can factor $\gcd(r,s)$ out of both $r,s$ so it suffices to prove the statement for $(nr,ns)$ in the case that $r,s$ are coprime.

We begin with the base cases.

\[\begin{split}
		(0r,0s)_{Y} \cdot w & = 2 \varnothing  \cdot w\\
				    & = 2w\\
				    & = (-1)^{0}2w,\\
		(r,s)_{Y} \cdot w & = Y_{1}(r,s) \cdot w\\
				  & = ((r,s) - \varnothing) \cdot w\\
				  & = (r,s) \cdot w - \varnothing \cdot w\\
				  & = -w - w\\
				  & = -2w\\
				  & = (-1)^{1}2w,\\
		(2r,2s)_{Y} \cdot w & = Y_{2}(r,s) \cdot w\\
				    & = ((r,s)^{2} - 2(r,s) - \varnothing) \cdot w\\
				    & = (r,s)^{2} \cdot w - 2 (r,s) \cdot w - \varnothing \cdot w\\
				    & = -(r,s) \cdot w + 2 w - w\\
				    & = w + 2w - w\\
				    & = 2w\\
				    & = (-1)^{2}2w.\\
	\end{split}\]

Proceeding by induction, let $n \ge 3$ such that $(kr,ks)_{Y} \cdot w = (-1)^{k}2w$ for all $0 \le k < n$. Then
	\[\begin{split}
		(nr,ns)_{Y} \cdot w & = Y_{n}(r,s) \cdot w\\
				    & = ((r,s) \cdot Y_{n - 1}(r,s) - (r,s) \cdot Y_{n - 2}(r,s) + Y_{n - 3}(r,s)) \cdot w\\
                    & = (r,s) \cdot Y_{n - 1}(r,s) \cdot w - (r,s) \cdot Y_{n - 2}(r,s) \cdot w + Y_{n - 3}(r,s) \cdot w\\
				    & = (r,s) \cdot (-1)^{n - 1}2w - (r,s) \cdot (-1)^{n - 2}2w + (-1)^{n - 3}2w\\
				    & = (-1)^{n}2w - (-1)^{n - 1}2w + (-1)^{n - 3}2w\\
				    & = (-1)^{n}2w.\\
	\end{split}\]
	The desired result follows by induction. Since $w$ commutes with any $(r,s)$, it commutes with $(nr,ns)_{Y}$, giving us the second equality.

\end{proof}

\begin{lemma}\label{product of threaded curves}
\begin{equation*}
(k,l)_{Y} \cdot (r,s)_{Y} = q^{-2 \begin{vmatrix} k & l \\ r & s \end{vmatrix}} (k + r, l + s)_{Y} + q^{2 \begin{vmatrix} k & l \\ r & s \end{vmatrix}} (k - r, l - s)_{Y} + \alpha w
\end{equation*}
\end{lemma}

\begin{proof}
The key idea is that the $JW$ map restricts to an injective linear map on the set $\mathrm{span}\{\beta \setminus \{K\}\}$ spanned by basis multicurves. Thus, we can use it, along with the Frohman-Gelca formula, to compute $(k,l)_{Y} \cdot (r,s)_Y$ up to the addition of an element $\alpha w \in \ker(JW).$ The main insight is that we can use Lemmas \ref{w^2} and \ref{w*(r,s)_Y} to compute the coefficient $\alpha.$ We have

\[\begin{split}
		JW((k,l)_{Y} \cdot (r,s)_{Y}) & = JW((k,l)_{Y}) \cdot JW((r,s)_{Y})\\
					      & = (2k,2l)_{T} \cdot (2r,2s)_{T}\\
					      & = q^{-\frac{1}{2} \begin{vmatrix} 2k & 2l\\ 2r & 2s \end{vmatrix}}(2k + 2r,2l + 2s)_{T} + q^{\frac{1}{2} \begin{vmatrix} 2k & 2l\\ 2r & 2s \end{vmatrix}}(2k - 2r,2l - 2s)_{T}\\
					      & = q^{-2 \begin{vmatrix} k & l \\ r & s \end{vmatrix}} (2(k + r),2(l + s))_{T} + q^{2 \begin{vmatrix} k & l \\ r & s \end{vmatrix}} (2(k - r),2(l - s))_{T}\\
					      & = q^{-2 \begin{vmatrix} k & l \\ r & s \end{vmatrix}} JW((k + r, l + s)_{Y}) + q^{2 \begin{vmatrix} k & l \\ r & s \end{vmatrix}} JW((k - r, l - s)_{Y}).\\
	\end{split}\]
	Thus,
	\[\begin{split}
		(k,l)_{Y} \cdot (r,s)_{Y} & \in q^{-2 \begin{vmatrix} k & l \\ r & s \end{vmatrix}} (k + r, l + s)_{Y} + q^{2 \begin{vmatrix} k & l \\ r & s \end{vmatrix}} (k - r, l - s)_{Y} + \mathrm{ker}(JW)\\
					  & = q^{-2 \begin{vmatrix} k & l \\ r & s \end{vmatrix}} (k + r, l + s)_{Y} + q^{2 \begin{vmatrix} k & l \\ r & s \end{vmatrix}} (k - r, l - s)_{Y} + \mathrm{span}(w).\\
	\end{split}\]
	Therefore, we can write
	\[(k,l)_{Y} \cdot (r,s)_{Y} = q^{-2 \begin{vmatrix} k & l \\ r & s \end{vmatrix}} (k + r, l + s)_{Y} + q^{2 \begin{vmatrix} k & l \\ r & s \end{vmatrix}} (k - r, l - s)_{Y} + \alpha w,\]
	for some $\alpha \in \mathbb{C}$. To compute such a coefficient, first multiply both sides of the equation by $w$:
    
	\[\begin{split}
		& w \cdot \left( (k,l)_{Y} \cdot (r,s)_{Y} \right) = w \cdot \left( q^{-2 \begin{vmatrix} k & l \\ r & s \end{vmatrix}} (k + r, l + s)_{Y} + q^{2 \begin{vmatrix} k & l \\ r & s \end{vmatrix}} (k - r, l - s)_{Y} + \alpha w \right).\\
    \end{split}\]

    Using Lemma \ref{w*(r,s)_Y} twice on the left-hand side, and using Lemmas \ref{w*(r,s)_Y} and \ref{w^2} on the right-hand side, we obtain

    \[\begin{split}
		 & (-1)^{g_{1} + g_{2}}4w = q^{-2 \begin{vmatrix} k & l \\ r & s \end{vmatrix}} (-1)^{g_{3}}2w + q^{2 \begin{vmatrix} k & l \\ r & s \end{vmatrix}} (-1)^{g_{4}}2w + 2[2]^{2} \alpha w.\\
    \end{split}\]

    We then equate coefficients of $w$ to solve for $\alpha$ and obtain

    \[\begin{split}
		& \alpha = \frac{(-1)^{g_{1} + g_{2}}2 - q^{-2 \begin{vmatrix} k & l \\ r & s \end{vmatrix}} (-1)^{g_{3}} - q^{2 \begin{vmatrix} k & l \\ r & s \end{vmatrix}} (-1)^{g_{4}}}{[2]^{2}},\\
	\end{split}\]
	as desired.
\end{proof}

\begin{corollary}\label{zero divisor}
The algebra $\Y(\Sigma_{1,0})$ is not a domain and thus there does not exist an injective algebra homomorphism $\Y(\Sigma_{1,0}) \rightarrow \Sq(\Sigma_{1,0}).$
\end{corollary}

\begin{proof}
The element $w$
is a nontrivial linear combination of basis elements from Theorem \ref{Basis torus}, so it is nonzero. The equation $w^2=2[2]^2w$ from Lemma \ref{w^2} shows that $w$ is a zero divisor.

Since $\Sq(\Sigma_{1,0})$ is a domain by \cite{FG00}, we see that no algebra embedding $\Y(\Sigma_{1,0}) \hookrightarrow \Sq(\Sigma_{1,0})$ could exist.
\end{proof}

\begin{corollary}\label{positive torus basis}
The algebra $\Y(\Sigma_{1,0})$ is isomorphic to $\mathbb{C} \times \Sq(\Sigma_{1,0})^{ev}$, where $\Sq(\Sigma_{1,0})^{ev}$ is the even subalgebra of $\Sq(\Sigma_{1,0})$, generated by those diagrams whose intersection number mod $2$ with any simple curve is $0$. In particular, the algebra $\Y(\Sigma_{1,0})$ has a positive basis $P$ with structure constants in $\mathbb{N}[q,q^{-1}]$ given by $P=\{e\}\cup\{c_{r,s} \mid r,s \in \mathbb{Z}\}$ where $e=\frac{1}{2[2]^2}w$ and $c_{r,s}=(1-e)(r,s)_Y.$ The product of pairs of elements of this basis is given by:
\[\begin{split}
	c_{k,l} \cdot c_{r,s} & = q^{2(rl-ks)}c_{k+r,s+l} + q^{-2(rl-ks)}c_{k-r,s-l},\\
	c_{k,l} \cdot e & =0= e \cdot c_{k,l},\\
	e^2 & = e.
\end{split}\]
\end{corollary}

\begin{proof}
The linear map from $\Y(\Sigma_{1,0})$ to itself defined on the basis $\{w\} \cup \{(r,s)_{Y} : r,s \in \mathbb{Z}\}$ by $w \mapsto e, (r,s)_{Y} \mapsto (1 - e)(r,s)_{Y}$ is upper triangular with respect to this basis, with $1$ on the diagonal everywhere except the first entry, which is nonzero. Thus, the image of this basis, $P$, is also a basis. The multiplication identities for pairs of $P$ follow by applying the identities in Theorem \ref{torus presentation}. We have a chain map of exact sequences
\[\begin{tikzcd}
0 \arrow[r] & \left< e \right> \arrow[r, "\iota"] \arrow[d, "{[-]_{\{e\}}}"] & \Y(\Sigma_{1,0}) \arrow[r, "{(1 - e) \cdot -}"] \arrow[d, equal] & \left< 1-  e \right> \arrow[r] \arrow[d, "JW"] & 0\\
0 \arrow[r] & \mathbb{C} \arrow[r, "{e \cdot -}"] & \Y(\Sigma_{1,0}) \arrow[r, "JW"] & \Sq(\Sigma_{1,0})^{ev} \arrow[r] & 0
\end{tikzcd}\]
where $[-]_{\{e\}}$ is the coordinate map of the basis $\{e\}$. The vertical map $JW$ is an isomorphism since the other two vertical maps are. Since $e$ is a central idempotent, the top sequence is split, so the bottom must be too. That is, $\Y(\Sigma_{1,0}) \cong \mathbb{C} \times \Sq(\Sigma_{1,0})^{ev}$. Note that the basis $P$ corresponds to the basis $\{(1,0)\} \cup \{(0,(2r,2s)_{T}) : r,s \in \mathbb{Z}\}$ of $\mathbb{C} \times \Sq(\Sigma_{1,0})^{ev}$.
\end{proof}

\subsection{Center of the $\rm{PGL}_2$ skein algebra of the torus}
We can now compute the center of $\Y(\Sigma_{1,0}).$

\begin{theorem}\label{torus center}
\phantom\\
\begin{enumerate}[(a)]
\begin{item}
If $q$ is not a root of unity, then the center of $\Y(\Sigma_{1,0})$ is as follows:
\begin{equation*}
Z(\Y(\Sigma_{1,0})) = \mathrm{span}\{\varnothing,w\}.
\end{equation*}
\end{item}
\begin{item}
If $q=\omega_n$ is a primitive $n^{\text{th}}$ root of unity, then the center of $S_{\omega_{n}}^{Y}(\Sigma_{1,0})$ is
\begin{equation*}
Z(S_{\omega_{n}}^{Y}(\Sigma_{1,0})) = \mathrm{span}(\{(k,l)_{Y} : n \mid 4\mathrm{gcd}(k,l)\} \cup \{w\}).
\end{equation*}
\end{item}
\end{enumerate}
\end{theorem}

\begin{proof}
	We will use the basis $\{w\} \cup \{(k,l)_{Y} : l \ge 0\}$ for $\Y(\Sigma_{1,0})$. We note the element $\varnothing$ is central and by Theorem \ref{torus presentation}, $w$ is central. As before, take
	\[g_{1} = \mathrm{gcd}(k,l), g_{2} = \mathrm{gcd}(r,s), g_{3} = \mathrm{gcd}(k + r, l + s), g_{4} = \mathrm{gcd}(k - r, l - s).\]
	We have
	\begin{align*}
		(k,l)_{Y} \cdot (r,s)_{Y} & = q^{-2 \begin{vmatrix} k & l \\ r & s \end{vmatrix}} (k + r, l + s)_{Y} + q^{2 \begin{vmatrix} k & l \\ r & s \end{vmatrix}} (k - r, l - s)_{Y}\\
					  & \quad + \frac{(-1)^{g_{1} + g_{2}}2 - (-1)^{g_{3}} q^{-2 \begin{vmatrix} k & l \\ r & s \end{vmatrix}} -  (-1)^{g_{4}} q^{2 \begin{vmatrix} k & l \\ r & s \end{vmatrix}}}{[2]^{2}} w,\\
	\end{align*}
	and $(r,s)_Y \cdot (k,l)_{Y}$ has an identical formula except with $q$ replaced by $q^{-1}.$ Therefore, we have that the commutator is
	\[[(k,l)_{Y},(r,s)_{Y}] = \left( q^{-2 \begin{vmatrix} k & l \\ r & s \end{vmatrix}} - q^{2 \begin{vmatrix} k & l \\ r & s \end{vmatrix}} \right) \left((k + r, l + s)_{Y} - (k - r, l - s)_{Y} - \frac{1}{[2]^{2}}((-1)^{g_{3}} + (-1)^{g_{4}})w \right).\]
	 From the formula for the commutator and the fact that $w$ is central, we have the inclusions
	\begin{align*}
		\mathrm{span}\{(0,0)_{Y},w\} = \mathrm{span}\{\varnothing, w\} & \subseteq Z(\Y(\Sigma_{1,0})),\\
		\mathrm{span}(\{w\} \cup \{(k,l)_{Y} : n \mid 4\mathrm{gcd}(k,l)\}) & \subseteq Z(S_{\omega_{n}}^{Y}(\Sigma_{1,0})).\\
	\end{align*}
	For the reverse inclusion, we use the standard argument using highest term. We place an order on torus knots threaded by $Y_{n}$: $(k,l)_{Y} > (r,s)_{Y}$ if $k>r$ or if $k \geq r$ and $l>s$, where representatives $(k,l),(r,s)$ are determined similarly to the convention given in Section \ref{torus section}. Then if 
	\[a = cw +b \varnothing+ \sum_{i = 1}^{m} c_{i}(k_{i},l_{i})_{Y} \in \Y(\Sigma_{1,0})\]
	is central with $(k_i,l_i)<(k_j,l_j)$ for $i<j$ and for $c_m \neq0,$ we must have that for all pairs $(r,s)$, $[(k_m,l_m)_Y,(r,s)_Y] \in \mathrm{span}\{w\}$. Let $g=\gcd(k_m,l_m)$ and choose $(r,s)$ so that $k_{m}r - l_{m}s = g.$ Then 
	\[[(k_m,l_m)_Y,(r,s)_{Y}] = \left( q^{-2g} - q^{2g}\right) \left((k_{m} + r, l_{m} + s)_{Y} - (k_{m} - r, l_{m} - s)_{Y} - \frac{1}{[2]^{2}}((-1)^{g_{3}} + (-1)^{g_{4}})w \right)\]
    is in $\mathrm{span}\{w\}$. Since $\{(k,l)_{Y} : k,l \in \mathbb{Z}\} \cup \{w\}$ is a basis, we must have $q^{-2g} - q^{2g} = 0$. If $q$ is generic, then $g=0,$ so $a \in \mathrm{span}(w,\varnothing).$ If $q$ is a primitive root of unity of order $n,$ we must have $q^{4g}=1$ or that $n \mid 4g.$ Since $a$ and $c_m(k_m,l_m)$ are both central in this case, we see that $a-c_m(k_m,l_m)$ is central as well and the above argument may be repeated. Therefore, $a \in \mathrm{span}(\{w\} \cup\{(k,l)_{Y} : n \mid 4 \mathrm{gcd}(k,l)\})$. We then have the desired reverse inclusions, so
	\begin{align*}
		Z(\Y(\Sigma_{1,0})) & = \mathrm{span}\{\varnothing, w\} & (\text{$q$ generic}),\\
		Z(S_{\omega_{n}}^{Y}(\Sigma_{1,0})) & = \mathrm{span}(\{(k,l)_{Y} : n \mid 4 \mathrm{gcd}(k,l)\} \cup \{w\}) & (\text{$\omega_{n}$ a primitive $n^{\text{th}}$ root of unity}).\\
	\end{align*}

\end{proof}

\begin{remark}\label{center comparison}
We may compare our center to the center in the Kauffman bracket case. The corresponding description of the center for the Kauffman bracket skein algebra is as follows. Let $q$ be a primitive $n^{th}$ root of unity. Then the center is
\[Z(\Sq(\Sigma_{1,0})) = \mathrm{span}\{(k,l)_{T} : n \mid \mathrm{gcd}(k,l)\}.\]
In terms of the Frobenius homomorphism, let $q^{2}$ be a primitive $n^{th}$ root of unity. Let $\epsilon = q^{n^{2}}$. It was computed in \cite{FKBL19} that
\[Z(\Sq(\Sigma_{1,0})) = \begin{cases}
\mathrm{Fr}(\mathcal{S}_{\epsilon}(\Sigma_{1,0})^{ev}) & \text{if} \ n \equiv 0 \ \text{mod} \ 2,\\
\mathrm{Fr}(\mathcal{S}_{\epsilon}(\Sigma_{1,0})) & \text{if} \ n \equiv 1 \ \text{mod} \ 2,\\
\end{cases}\]
where $\mathcal{S}_{\epsilon}^{ev}(\Sigma)$ is the subalgebra of $\mathcal{S}_{\epsilon}(\Sigma)$ generated by those diagrams whose intersection number mod $2$ with any simple curve is $0$.\\
For the $\rm{PGL}_2$ skein algebra, the Frobenius homomorphism $\rm{Fr}^Y$ for the torus was constructed in \cite{HSSSSW26}. If $q^{2}$ is a primitive $n^{\text{th}}$ root of unity,
\[\mathrm{Fr}^Y(\mathcal{S}_{\epsilon}^{Y}(\Sigma_{1,0})) =
\begin{cases}
Z(\Y(\Sigma_{1,0}))^{ev^{Y}} & \text{if} \ n \equiv 0 \ \text{mod} \ 2,\\
Z(\Y(\Sigma_{1,0})) & \text{if} \ n \equiv 1 \ \text{mod} \ 2\\
\end{cases}\]
where $Z(\Y(\Sigma_{1,0}))^{ev^{Y}}$ is the subalgebra generated by $\{(2k,2l)_{Y} : (k,l)_{Y} \in Z(\Y(\Sigma_{1,0}))\}$. In particular, for $n \equiv 0$ mod $2$, the center properly contains the image of the Frobenius homomorphism since $w \notin \mathrm{Fr}(\mathcal{S}_{\epsilon}^{Y}(\Sigma_{1,0}))$, by \cite{HSSSSW26}.
\end{remark}

\section{The once-punctured torus}\label{1-punctured torus section}

In this section, we will construct a basis and algebra presentation for the $\rm{PGL}_2$ skein algebra of the once-punctured torus, $\Y(\Sigma_{1,1})$. The basis and presentation will borrow some names of the diagrams involved with the closed torus.

As in the case of the closed torus, we will define $K \in \Y(\Sigma_{1,1})$ to be the element represented by the same underlying diagram as $K \in \Y(\Sigma_{1,0}).$ The choice of placement of the puncture in the complement of $K$ does not change the element since the complement of $K$ is a disk. Also as in the case of the closed torus, we will define $(k,l) \in \Y(\Sigma_{1,1})$ for $k,l$ coprime to be the element represented by the torus knot. Here as well, the choice of placement of the puncture does not affect the element. Finally, we let $w \in \Y(\Sigma_{1,1})$ represent the lift of the element $w \in \Y(\Sigma_{1,0}).$

We next introduce the diagrams that will be specific to our study of the once-punctured torus. Let $b_{k,l}$ be the diagram obtained by taking the $(k,l)$ torus knot and attaching a bigon around the puncture of $\Sigma_{1,1}$. Let $x$ be the boundary-parallel curve around the puncture of $\Sigma_{1,1}$.

\subsection{A basis for the once-punctured torus}

\begin{theorem}\label{punctured torus basis}
	The set 
	\[\alpha = \{x^{m}b_{k,l}^{\delta}(k,l)^{n} : \mathrm{gcd}(k,l) = 1, m,n \in \mathbb{Z}_{\ge 0}, \delta \in \{0,1\}\} \cup \{x^{m}w : m \in \mathbb{Z}_{\ge 0}\}\]
	 is a basis of $\Y(\Sigma_{1,1})$.
\end{theorem}

\begin{proof}
The proof follows from Lemmas \ref{punctured torus linear independence} and \ref{punctured torus spanning set} below.
\end{proof}

To aid us, we will frequently use the map $JW: \Y(\Sigma_{1,1}) \rightarrow \Sq(\Sigma_{1,1})$ which is injective by Corollary \ref{JW injective on punctured torus}. We record the following helpful data.

\begin{lemma}\label{JW of generators}
The following computations hold for the map $JW : \Y(\Sigma_{1,1}) \to \Sq(\Sigma_{1,1})$:
\begin{align*}
JW(x) & = x^{2} - \varnothing,\\
JW(b_{k,l}) & = \frac{1}{[2]}(x^2-\varnothing)+\frac{2}{[2]}((k,l)^2-\varnothing)+\frac{1}{[2]}\varnothing+x(k,l)^2,\\
JW((k,l)) & = (k,l)^{2} - \varnothing,\\
JW(w) & = [2]x + [2]^{2}\varnothing.\\
\end{align*}
\end{lemma}

\begin{proof}
The identities follow by direct diagrammatic computation.
\end{proof}

\begin{lemma}\label{punctured torus linear independence}
The set $\alpha$ is linearly independent in $\Y(\Sigma_{1,1}).$    
\end{lemma}

\begin{proof}
To see that $\alpha$ is linearly independent, we apply the algebra map $JW: \Y(\Sigma_{1,1}) \rightarrow \Sq(\Sigma_{1,1})$ to our set $\alpha.$ First, we rewrite $\alpha$ as

\begin{equation*}
\alpha=\{x^m(k,l)^n \} \cup \{x^mb_{k,l}(k,l)^n\}\cup\{x^mw\}.
\end{equation*}

Then, using Lemma \ref{JW of generators}, $JW(\alpha)$ is given by

\begin{equation*}
JW(\alpha)=A\cup B \cup C,
\end{equation*}

where
\begin{align*}
A&=\{(x^2-\varnothing)^m((k,l)^2-\varnothing)^n\}\\
B&=\{(x^2-\varnothing)^m\left(\frac{1}{[2]}(x^2-\varnothing)+\frac{2}{[2]}((k,l)^2-\varnothing)+\frac{1}{[2]}\varnothing+x(k,l)^2\right)\left((k,l)^2-\varnothing\right)^n\}\\
C&=\{([2]x+[2]^2)(x^2-\varnothing)^{m}\}.
\end{align*}

By using the multicurve basis of $\Sq(\Sigma_{1,1})$ we may deduce that $JW$ maps $\alpha$ injectively to a linearly independent set in $\Sq(\Sigma_{1,1})$. Therefore, $\alpha$ is linearly independent.
\end{proof}

\begin{lemma}\label{punctured torus spanning set}
The set $\alpha$ is a spanning set of $\Y(\Sigma_{1,1}).$
\end{lemma}

\begin{proof}
First, we will classify connected components and show they can be written as a linear combination of $\alpha$.

Curves on the punctured torus are either $(k,l)$ for $k,l$ coprime, or $x$, and thus are in the span of $\alpha.$ For cubic graphs, we can consider subcases based on the rank of the graph viewed as a subgraph of the closed torus (as defined in the proof of Lemma \ref{span torus}).
    
First, suppose that the graph has rank zero. Then there exists an embedding of a disk or annulus, depending on whether the puncture is present, which is surjective onto the graph. Then we can reduce the graph to a multiple of $\varnothing$ or a polynomial in $x$, respectively.
    
Next, if the graph has rank one, then there exists an embedding of an annulus or three-punctured sphere (thought of as a punctured annulus) with two boundary components parallel to a $(k,l)$ curve, for some $k,l$ coprime, which is surjective onto the graph. In the former case, we can reduce the graph to a polynomial in the $(k,l)$ curve. In the latter, the web is contained in a neighborhood diffeomorphic to the 3-punctured sphere, and so it can be reduced to a linear combination of elements of the form  $x^mb_{k,l}^{\delta}(k,l)^n$ by \cite{Chb12}.

If the graph has rank two with at least four vertices, we will describe how to reduce it to a linear combination of $K$ and graphs with rank zero or rank one, which are handled above. A slight modification of the proof of Lemma \ref{span torus} will be made to take into account the puncture. We note the following sequence of I-H flips.
\[\modifiedih \sim \modifiedihi \sim \modifiedihii \sim \modifiedihiii \sim \modifiedihiv.\] 
Thus, an analog of an I-H move holds on a bubbled punctured edge. For a rank-two graph with $2n$ vertices for $n \geq 2$, we then have that our diagram is I-H equivalent to:
\[\ranktwomonogonpunctured,\]
which is zero when $n > 2$. For $n = 2$, our graph is equivalent to the above diagram with no monogons, and we can perform the following sequence of I-H moves:
\[\bigonhexpuncturedi \sim \bigonhexpuncturedii \sim \bigonhexpuncturediii \sim \bigonhexpuncturediv \sim \bigonhexpuncturedv \sim \bigonhexpuncturedvi.\]
Now, using the triangle relation in Equation (\ref{ngon formulae}), we have
\[\bigonhexpuncturedvi = \left( \frac{2 - [3]}{[2]} \right) \bigonhexpuncturedvii = \left( \frac{2 - [3]}{[2]} \right) \hexpunctured = \left( \frac{2 - [3]}{[2]} \right) K.\]
When $n = 1$, our  rank-two graph is I-H equivalent to $K.$ Thus, we can apply induction to see all rank-two graphs reduce to linear combinations of elements of $\alpha$. Thus, all connected components are linear combinations of elements of $\alpha$.
\end{proof}

\subsection{An algebra presentation for the once-punctured torus}

In this section, we will find a finite presentation for $\Y(\Sigma_{1,1}).$ We begin with the following proposition, which will allow us to show that $\Y(\Sigma_{1,1})$ is finitely generated by following an adaptation of the argument for the Kauffman bracket case given in \cite{BP00}.

\begin{proposition}\label{qbar induction}
	Let $X \subseteq \mathbb{Q} \cup \{\infty\}$. Suppose that $X$ satisfies the following properties: 
	\begin{enumerate}[(i)]
		\begin{item}
			$0,1,\infty \in X$,
		\end{item}
		\begin{item}
			for all $k,l,r,s \in \mathbb{Z}$ such that $k/l,r/s \in X$ and
			\[\mathrm{det} \begin{pmatrix} k & l \\ r & s \end{pmatrix} = 1,\]
            we have
			\[\frac{k + r}{l + s} \in X \iff \frac{k - r}{l - s} \in X.\]
		\end{item}
	\end{enumerate}
	Then $X = \mathbb{Q} \cup \{\infty\}$.
\end{proposition}
\begin{proof}
We make use of the following construction. Define the set
\[M = \left\{ \begin{pmatrix} a & b \\ c & d \end{pmatrix} \in \mathrm{SL}_{2}(\mathbb{Z}) : \frac{a}{b}, \frac{c}{d}, \frac{a + c}{b + d}, \frac{a - c}{b - d} \in X \right\}.\]
We first claim $M = \mathrm{SL}_{2}(\mathbb{Z})$. Recall that
\[S = \begin{pmatrix} 0 & -1 \\ 1 & 0 \end{pmatrix}, T = \begin{pmatrix} 1 & 1 \\ 0 & 1 \end{pmatrix}\]
generate $\mathrm{SL}_{2}(\mathbb{Z})$. We will show by induction on word length that $M = \mathrm{SL}_{2}(\mathbb{Z})$.

First, it may be shown that the identity matrix $I \in \mathrm{SL}_{2}(\mathbb{Z})$ satisfies $I \in M$ since $1/0,0/1,1/1 \in X$ by the first property and $1/(-1) \in X$ by the second property.\\
Next, let $A \in M$. We claim that $S^{\pm 1}A,T^{\pm 1}A \in M$. If such a matrix $B = S^{\pm 1}A$ or $B = T^{\pm 1}A$ is of the form
\[B = \begin{pmatrix} a & b \\ c & d \end{pmatrix},\]
it can be shown that $a/b, c/d \in X$ since $A \in M$. Additionally, if $B = S^{\pm 1}A$ or $B = T^{-1}A$, then $\frac{a + c}{b + d} \in X$ since $A \in M$ and so $\frac{a-c}{b-d} \in X$ by the second property of $X.$
Similarly, if $B = TA$, then $\frac{a - c}{b - d} \in X$ since $A \in M$ and so $\frac{a+c}{b+d} \in X$ by the second property of $X.$ Thus, in each case, we have $B \in M$.
It follows by induction that all words in $S^{\pm 1}$ and $T^{\pm 1}$ belong to $M$, so that $M = \mathrm{SL}_{2}(\mathbb{Z})$.

We may now prove the proposition. Let $k/l \in \mathbb{Q} \cup \{\infty\}$. Without loss of generality, assume $\mathrm{gcd}(k,l) = 1$; otherwise, we may reduce the fraction. By B\'{e}zout's identity, there exist $r,s \in \mathbb{Z}$ such that $kr + ls = 1$. The matrix
\[\begin{pmatrix} k & l \\ -s & r \end{pmatrix}\]
has determinant $kr + ls = 1$, so it is in $\mathrm{SL}_{2}(\mathbb{Z})$, and thus in $M$. In particular, then, its first row has associated fraction $k/l \in X$.
\end{proof}

\begin{theorem}\label{punctured torus presentation}
	  If $q^4 \neq 1,$ then $\Y(\Sigma_{1,1})$ has a presentation
	\[\Y(\Sigma_{1,1}) \cong \left< a,b,c,h \mid R \right>,\]
	with an isomorphism $f:\left< a,b,c,h \mid R \right> \rightarrow \Y(\Sigma_{1,1})$ given by 
	\[a \mapsto (0,1) + \varnothing, b \mapsto (1,0) + \varnothing, c \mapsto (1,1) + \varnothing, h \mapsto \frac{1}{[2]} w + q(0,1) + q^{-1}(1,0) + q(1,1) - q^{-1} \varnothing,\]
	the relations $R$ are given by
    \[\begin{split}
        abc = & -(q^{-6} - 2q^{-4} - q^{-2} + 4 - q^{2} - 2q^{4} + q^{6})c - (2q^{-5} - 3q^{-3} - 2q^{-1} + 4q - q^{5})h\\
			      & - (q^{-2} - 2 + q^{2})ac - (q^{-4} - 2 + q^{4})bc\\
			      & - (q^{-1} - q)ah + (q^{-5} - q^{-3})bh\\
			      & - (q^{-1} - q)ch + (q^{-2})h^{2},\\
    \end{split}\]
	the $q^{2}$-deformed commutators of pairs of $a,b,c$ given by
	\[\begin{split}
		[a,b]_{q^{2}} & = \left( \begin{split}
			& -(q^{-2} - 1 - q^{2} + q^{4})a + (q^{-4} - q^{-2} - 1 + q^{2})b\\
			& - (q^{-4} - q^{-2} - 1 + q^{2})c - (q^{-3} + q^{-1} - q - q^{3})h\\
		\end{split} \right),\\
		[b,c]_{q^{2}} & = \left( \begin{split}
			& -(q^{-4} - q^{-2} - 1 + q^{2})a + (q^{-4} - q^{-2} - 1 + q^{2})b\\
			& - (q^{-2} - 1 - q^{2} + q^{4})c - (q^{-3} + q^{-1} - q - q^{3})h\\
		\end{split} \right),\\
		[c,a]_{q^{2}} & = \left( \begin{split}
			& -(q^{-2} - 1 - q^{2} + q^{4})a + (q^{-2} - 1 - q^{2} + q^{4})b\\
			& - (q^{-2} - 1 - q^{2} + q^{4})c - (q^{-3} + q^{-1} - q - q^{3})h
		\end{split} \right),\\
	\end{split}\]
	along with the commutators of $a,b,c$ with $h$ given by
	\[\begin{split}
		[a,h] & = \left( \begin{split}
			& (q^{-3} - 2q^{-1} + 2q^{3} - q^{5})a + (q^{-4} - 2 + q^{4})h\\
			& + (q^{-1} - q^{3})ab - (q^{-3} - q)ac\\
		\end{split} \right),\\
		[b,h] & = \left( \begin{split}
			& -(q^{-1} - 2q + 2q^{5} - q^{7})a + (q^{-1} - 2q + 2q^{5} - q^{7})c\\
			& - (q - q^{5})ab + (q - q^{5})bc\\
		\end{split} \right),\\
		[c,h] & = \left( \begin{split}
			& -(q^{-3} - 2q^{-1} + 2q^{3} - q^{5})c - (q^{-4} - 2 + q^{4})h\\
			& + (q^{-3} - q)ac - (q^{-1} - q^{3})bc\\
		\end{split} \right).\\
	\end{split}\]
\end{theorem}

\begin{proof}

To show that $f$ is an isomorphism of algebras we must show that $f$ gives a well-defined algebra homomorphism which is surjective and injective.

To see that $f$ is well-defined, we must check that each relation $r \in R$ gets mapped to an identity $f(r)$ which holds in $\Y(\Sigma_{1,1}).$ This could be done by a direct computation, but it is more convenient to use the map $JW: \Y(\Sigma_{1,1}) \rightarrow \Sq(\Sigma_{1,1}),$ which is injective by Proposition \ref{JW injective on pants}, to see that each relation $r$ gets mapped to an identity $JW(f(r))$ which holds in $\Sq(\Sigma_{1,1}).$ For this, we recall the presentation of \cite{BP00} given by

\[\Sq(\Sigma_{1,1}) \cong \left< x,y,z \mid  [x,y]_{q^{1/2}} = \delta z, [y,z]_{q^{1/2}} = \delta x, [z,x]_{q^{1/2}} = \delta y \right>,\]
	with $\delta = q-q^{-1}.$

One may compute that $JW(f(a))=x^2,$ $JW(f(b))=y^2,$ $JW(f(c))=z^2$ and 
$JW(f(h))=q^{1/2}xyz.$ Then the presentation of \cite{BP00} may be used to verify for each relation $r \in R,$ $JW(f(r))$ is an identity in $\Sq(\Sigma_{1,1}).$

To show $f$ is surjective, we will actually show that the elements $f(a),f(b),f(c)$ generate $\Y(\Sigma_{1,1}).$ The use of $f(h)$ is redundant as a generator but convenient for our presentation. Thus, we will show that the subalgebra $\left< (1,0), (0,1), (1,1) \right> = \Y(\Sigma_{1,1})$. In this part, we use the extra assumption that $q^4 \neq 1$ so that we may divide by $q^2 -q^{-2}\neq 0.$

Define a total order on torus knots, identified with their slopes in $\mathbb{Q} \cup \{\infty\}$, by extending the usual one on $\mathbb{Q}$ with the convention that for all $k/l \in \mathbb{Q} \cup \{\infty\}$, $k/l \le 1/0$. If $(k,l)$ and $(r,s)$ are two torus curves which meet exactly once, we may use the skein relations to compute their commutator as follows. Without loss of generality, take $k/l < r/s$. Then the following commutator relation is computed by applying the skein relation at the single crossing of the $(k,l)$ and $(r,s)$ curves:

\begin{equation}\label{commutator equation}
\frac{1}{q^{2} - q^{-2}}[(k,l),(r,s)] + (k - r,l - s)=(k + r, l + s).
\end{equation}
Let $X =\{(r,s) \in \left< (1,0), (0,1), (1,1) \right> \}$. Under identification of a torus knot with its slope in $\mathbb{Q} \cup \{\infty\}$, we claim $X$ satisfies the properties of Proposition \ref{qbar induction}. $X$ satisfies the first property since $(1,0), (0,1)$, and $(1,1) \in X.$ Now, given $(k,l),(r,s) \in \left< (1,0), (0,1), (1,1) \right>$ with
\[\mathrm{det} \begin{pmatrix} k & l \\ r & s \end{pmatrix} = 1,\]
$(k,l)$ and $(r,s)$ intersect exactly once. Then Equation (\ref{commutator equation}) gives that $(k - r, l - s) \in \left< (1,0), (0,1), (1,1) \right>$ if and only if $(k + r, l + s) \in \left< (1,0), (0,1), (1,1) \right>$. Therefore, by Proposition \ref{qbar induction}, $X = \mathbb{Q} \cup \{\infty\}$. That is, all torus knots are generated by $\{(1,0),(0,1),(1,1)\}$.

We next observe that the elements $K,x,b_{k,l}$ are generated by $(1,0), (0,1)$ and $(1,1)$. The following identities can be verified by applying the map $JW: \Y(\Sigma_{1,1})\rightarrow \Sq(\Sigma_{1,1})$, which is injective by Proposition \ref{JW injective on pants}, to both sides and reducing in the multicurve basis of $\Sq(\Sigma_{1,1}).$

For $K$, we have that

\begin{equation}\label{K generated}
K=\left( \frac{1}{[2]^{2}(q - q^{-1})} \right) \left( q^{2}((0,1) \cdot (1,0) - (q^{2} - 1)(1,1)) - q^{-2} ((1,0) \cdot (0,1) - (q^{-2} - 1)(1,1)) \right).
\end{equation}

The element $x$ which is represented by a curve parallel to the boundary of $\Sigma_{1,1}$ is generated according to the following equation:

\begin{equation*}
x=\frac{1}{[2]^{2}}([2]K - (0,1) - (1,0) - (1,1) - \varnothing )^{2} - \varnothing.
\end{equation*}

Lastly, we observe that each bigon curve $b_{k,l}$ is generated by multicurves according to the following equation:

\begin{equation*}
b_{k,l}=\frac{1}{[2]} \left(([2]K - (0,1) - (1,0) - (1,1) - \varnothing)((k,l) + \varnothing) + x + 2(k,l) + \varnothing \right).
\end{equation*}

We have now shown that each generator of the basis $\alpha$ of $\Y(\Sigma_{1,1})$ is generated by the image of $f,$ so $f$ is surjective.

To show that $f$ is injective, we will argue that $f$ maps a spanning set of $\langle a,b,c,h \mid R \rangle $ to a linearly independent set of $\Y(\Sigma_{1,1}).$ We note that the relations $R$ may be used to show that $\{a^ib^jc^kh^l \mid i,j,k,l \in \mathbb{Z}_{\geq 0}, ijk=0 \}$ is a spanning set. To show that $f(\{a^ib^jc^kh^l \mid i,j,k,l \in \mathbb{Z}_{\geq 0}, ijk = 0\})$ is linearly independent, it will suffice to observe that $S=JW(f(\{a^ib^jc^kh^l \mid i,j,k,l \in \mathbb{Z}_{\geq 0}, ijk = 0\}))$ is linearly independent in the Kauffman bracket skein algebra $\Sq(\Sigma_{1,1}).$ As noted earlier, $JW(f(a))=x^2,JW(f(b))=y^2,JW(f(c))=z^2$ and $JW(f(h))=q^{1/2}xyz.$ It then follows from the monomial basis of $\Sq(\Sigma_{1,1})$ from \cite[Theorem 2.1]{BP00} that $S$ is linearly independent. Therefore, we have completed the proof that $f$ is an isomorphism.

\end{proof}

\subsection{Kernel of the puncture-filling map}

We now describe the kernel of the puncture-filling map $\Y(\Sigma_{1,1}) \rightarrow \Y(\Sigma_{1,0})$ associated to filling in the puncture of the once-punctured torus to obtain the closed torus.

\begin{theorem}\label{puncture filling}
	The inclusion $i : \Sigma_{1,1} \to \Sigma_{1,0}$ induces an epimorphism of algebras
	\[i_{\ast} : \Y(\Sigma_{1,1}) \to \Y(\Sigma_{1,0}).\]
	We have that when $q^4 \neq 1,$
	\[\mathrm{ker}(i_{\ast}) =  \left< aw, bw, cw \right>,\]
	where $a = (0,1) + \varnothing, b = (1,0) + \varnothing, c = (1,1) + \varnothing$.
\end{theorem}

\begin{proof}
It is clear that $i_*$ is surjective, as every diagram on $\Sigma_{1,0}$ has a preimage on $\Sigma_{1,1}.$ The computation of the kernel follows from the following Lemma \ref{geometric kernel}, Lemma \ref{w times curve kernel}, and Lemma \ref{finitely generated kernel}.
\end{proof}

First, we give a description of the kernel in geometric terms: it is generated by elements which arise from basis webs which have reducible components upon filling in the puncture.
\begin{lemma}\label{geometric kernel}
The kernel of $i_{\ast}$ is
\[\ker(i_*)= \left<x-[3] \varnothing, b_{k,l}-\left(\frac{1-[3]}{[2]}\right)(k,l) : (k,l) \ \text{is a torus knot} \right>.\]
\end{lemma}
\begin{proof}
Consider the map $j : \Y(\Sigma_{1,0}) \to \Y(\Sigma_{1,1})$ given by sending basis diagrams of $\beta$ from Theorem \ref{Basis torus} to the corresponding diagram with a puncture. This is well-defined since the isotopy class is invariant with respect to the placement of the puncture. Define $P = j \circ i_{\ast}$. Since $i_{\ast} \circ j = \mathrm{id}$, $P$ is a projection and $j$ is injective. Then we have the vector space decomposition
\[\Y(\Sigma_{1,1}) = \mathrm{ker}(P) \oplus \mathrm{im}(P).\]
Since $j$ is injective, we have
\[\mathrm{ker}(i_{\ast}) = \mathrm{ker}(P) = \mathrm{im}(\mathrm{id} - P).\]
Using the decomposition of the basis $\alpha$ of $\Y(\Sigma_{1,1})$ as
\[\alpha = A' \cup B' \cup C'\]
where
\[A' = \{x^{m}(k,l)^{n} : m,n \ge 0\}, B' = \{x^{m}b_{k,l}(k,l)^{n} : m,n \ge 0\}, C' = \{x^{m}w : m \ge 0\}\]
from Lemma \ref{punctured torus linear independence} and mapping under $\mathrm{id} - P$, we see that
\[\mathrm{ker}(i_{\ast}) = \mathrm{span}((\mathrm{id} - P)(A' \cup B' \cup C')).\]
We may compute
\begin{align*}
(\mathrm{id} - P)(A') & = \left\{ x^{m}(k,l)^{n} - [3]^{m}(k,l)^{n} : m,n\geq 0 \right\},\\
(\mathrm{id} - P)(B') & = \left\{ x^{m}b_{k,l}(k,l)^{n} - [3]^{m}\left( \frac{1 - [3]}{[2]} \right) (k,l)^{n + 1} : m,n \ge 0 \right\}\\
(\mathrm{id} - P)(C') & = \left\{ x^{m}w - [3]^{m}w : m \ge 0 \right\}.
\end{align*}
Each element in the spanning set of $\ker(i_*)$ is in the ideal generated by $x - [3]\varnothing$ and $b_{k,l} - \left( \frac{1 - [3]}{[2]} \right) (k,l)$. We also can see that each generator of such an ideal is in the kernel. Thus, the kernel of $i_{\ast}$ is precisely the following ideal:
\begin{equation*}
\ker(i_*)= \left<x-[3] \varnothing, b_{k,l}-\left(\frac{1-[3]}{[2]}\right)(k,l)\right>.
\end{equation*}
\end{proof}
As our next intermediate step, we show the kernel of $i_{\ast}$ is generated by the following infinite set.
\begin{lemma}\label{w times curve kernel}
The kernel of $i_{\ast}$ is given by
\[\mathrm{ker}(i_{\ast}) = \left< w((k,l) + \varnothing) \right>,\] where $(k,l)$ ranges over all torus knots.
\end{lemma}
\begin{proof}
We first argue that $b_{k,l}-\left(\frac{1-[3]}{[2]}\right)(k,l) \in \langle x-[3] \varnothing, w\cdot((k,l)+\varnothing) \rangle.$
On one hand, we can compute that

\begin{align*}
JW(x - [3]\varnothing + w((k,l) + \varnothing))&=(x^{2} - \varnothing - [3] \varnothing) + ([2]x + [2]^{2}\varnothing)(k,l)^{2}.
\end{align*}

On the other hand, we can compute that

\begin{align*}
JW([2]b_{k,l} - (1 - [3])(k,l))&=x^{2} - \varnothing + 2((k,l)^{2} - \varnothing) + \varnothing + [2]x(k,l)^{2} - (1 - [3])((k,l)^{2} - \varnothing)\\
&=(x^{2} - \varnothing - [3] \varnothing) + ([2]x + [2]^{2}\varnothing)(k,l)^{2},
\end{align*}

where the last equality uses the identity $[3]=[2]^2-1.$

Since $JW$ is injective for $\Sigma_{1,1}$ by Corollary \ref{JW injective on punctured torus},
\[[2]b_{k,l} - (1 - [3])(k,l) = x - [3]\varnothing + w((k,l) + \varnothing),\]
so $b_{k,l} - \left( \frac{1 - [3]}{[2]} \right) (k,l) \in \left< x - [3] \varnothing, w \cdot ((k,l) + \varnothing) \right>$.

We also have the following computation, which shows that $x-[3]\varnothing \in \langle wa,wb,wc \rangle.$

We see
\[x = \frac{1}{[2]^{2}} (w - [2]^{2})^{2} - \varnothing,\]
by applying $JW$ to both sides using the computation from Lemma \ref{JW of generators}. Thus,
\begin{align*}
x - [3]\varnothing & = \frac{1}{[2]^{2}} (w - [2]^{2})^{2} - \varnothing - [3]\varnothing\\
& = \frac{1}{[2]^{2}}(w^{2} - 2[2]^{2}w + [2]^{4}\varnothing - [2]^{4}\varnothing)\\
& = \frac{1}{[2]^{2}}(w^{2} - 2[2]^{2}w).
\end{align*}
Now, it suffices to show that $w^{2} - 2[2]^{2}w \in \left< wa, wb, wc \right>$, which holds if $w - 2[2]^{2} \varnothing \in \left< a, b ,c \right>$. But writing $w$ diagrammatically as
\[w = [2]K - (1,0) - (0,1) - (1,1) + [3]\varnothing\]
and using Equation (\ref{K generated})
gives
\begin{align*}
w - 2[2]^{2} \varnothing = & \frac{1}{q^{2} - q^{-2}} \left( \begin{aligned}
& q^{2} \left( (0,1) \cdot (1,0) - (q^{2} - 1)(1,1) \right)\\
- & q^{-2} \left( (1,0) \cdot (0,1) - (q^{-2} - 1)(1,1) \right)
\end{aligned}
\right)\\
& - (1,0) - (0,1) - (1,1) + [3] \varnothing - 2[2]^{2}\varnothing.
\end{align*}
Now, the constant term after substituting $a,b,c$ for the three torus knots and expanding is given by substituting $-1$ for each torus knot. It suffices to check that the constant term is zero, so that each term in the expression is a multiple of $a$, $b$, or $c$. This is in fact the case; we have
\begin{align*}
& \frac{1}{q^{2} - q^{-2}} \left( \begin{aligned}
& q^{2} \left( (-1)^{2}  - (q^{2} - 1)(-1) \right)\\
- & q^{-2} \left( (-1)^{2} - (q^{-2} - 1)(-1) \right)
\end{aligned} \right) - (-1) - (-1) - (-1) + [3] - 2[2]^{2}\\
= & \frac{q^{4} - q^{-4}}{q^{2} - q^{-2}} + 3 + [3] - 2[2]^{2}\\
= & 2q^{2} + 4 + 2q^{-2} - 2(q^{2} + 2 + q^{-2})\\
= & 0.
\end{align*}
Thus, $w - 2[2]^{2}\varnothing \in \left< a,b,c \right>$, so $x - [3]\varnothing \in \left< wa, wb, wc \right>$.
\end{proof}
We next find our finite generating set for the kernel of $i_*$.
\begin{lemma}\label{finitely generated kernel}
The kernel of $i_{\ast}$ is finitely generated. In particular, it is given by the ideal
\[\mathrm{ker}(i_{\ast}) = \left< wa, wb, wc \right>.\]
\end{lemma}
\begin{proof}
The elements $wa,wb,wc$ may be written as $wa = w((0,1) + \varnothing), wb = w((1,0) + \varnothing), wc = w((1,1) + \varnothing)$, and thus are in the kernel of $i_*$ by Lemma \ref{w*(r,s)}. Thus, it suffices to show that the elements which generate $\ker(i_*)$ are in turn generated by $wa,wb,wc.$ This will employ a similar approach to Theorem \ref{punctured torus presentation} using Proposition \ref{qbar induction}, which showed that $(k,l) \in \langle a,b,c\rangle.$ Write
\[X = \{(k,l) \in \Y(\Sigma_{1,1}) : w((k,l) + \varnothing) \in \left< wa, wb, wc \right> \}.\]
Then $(1,0),(0,1),(1,1) \in X$ since
\[wa,wb,wc \in \left< wa, wb, wc \right>.\]
Thus, the first property of Proposition \ref{qbar induction} is met. We will now consider the second property. Suppose $(k,l)$ and $(r,s)$ are a pair of curves which intersect once, which holds whenever
\[\mathrm{det} \begin{pmatrix} k & l \\ r & s \end{pmatrix} = 1.\]
By multiplying both sides of Equation (\ref{commutator equation}) by $w$, we have:
\begin{equation*}
\frac{1}{q^{2} - q^{-2}}w[(k,l) + \varnothing,(r,s) + \varnothing] + w(k - r, l - s) = w(k + r,l + s)\\
\end{equation*}
From this, we can deduce that $w((k - r,l - s) + \varnothing) \in \left< wa,wb,wc \right>$ if and only if $w((k + r,l + s) + \varnothing) \in \left< wa,wb,wc \right>$. Then $X$ satisfies the second property as well, so by Proposition \ref{qbar induction}, $X = \mathbb{Q} \cup \{\infty\}$. That is, for all torus knots $(k,l)$, we have $w((k,l) + \varnothing) \in \left< wa, wb, wc \right>$. It follows $\text{ker}(i_{\ast}) = \left< wa,wb,wc \right>$.
\end{proof}

\begin{remark}
We observe that $w$ is central in $\Y(\Sigma_{1,1})$ by using the fact that $JW: \Y(\Sigma_{1,1}) \rightarrow \Sq(\Sigma_{1,1})$ is an injective map and observing that $JW(w)=[2]x+[2]^2\varnothing \in Z(\Sq(\Sigma_{1,1}))$. Thus, $Z(\Y(\Sigma_{1,1}))$ is not generated by the central boundary-parallel curve $x$ as was true for the case of the Kauffman bracket skein algebra.
\end{remark}

We finish by observing the following corollary of Theorem \ref{puncture filling}.

\begin{corollary}\label{finite presentation torus}
When $q^4 \neq 1,$ the algebra $\Y(\Sigma_{1,0})$ has a finite presentation given by the quotient of $\Y(\Sigma_{1,1})$ presented in Theorem \ref{punctured torus presentation} modulo the extra set of relations: $wa=wb=wc=0.$
\end{corollary}

\bibliographystyle{amsalpha}
\bibliography{references.bib}

@article {BP00,
    AUTHOR = {Bullock, Doug and Przytycki, J\'ozef H.},
     TITLE = {Multiplicative structure of {K}auffman bracket skein module
              quantizations},
   JOURNAL = {Proc. Amer. Math. Soc.},
  FJOURNAL = {Proceedings of the American Mathematical Society},
    VOLUME = {128},
      YEAR = {2000},
    NUMBER = {3},
     PAGES = {923--931},
      ISSN = {0002-9939,1088-6826},
   MRCLASS = {57M25},
  MRNUMBER = {1625701},
MRREVIEWER = {Quach thi C\^am V\^an},
       DOI = {10.1090/S0002-9939-99-05043-1},
       URL = {https://doi.org/10.1090/S0002-9939-99-05043-1},
}

@article {Prz91,
    AUTHOR = {Przytycki, J\'ozef H.},
     TITLE = {Skein modules of {$3$}-manifolds},
   JOURNAL = {Bull. Polish Acad. Sci. Math.},
  FJOURNAL = {Bulletin of the Polish Academy of Sciences. Mathematics},
    VOLUME = {39},
      YEAR = {1991},
    NUMBER = {1-2},
     PAGES = {91--100},
      ISSN = {0239-7269},
   MRCLASS = {57M25 (57N10)},
  MRNUMBER = {1194712},
}

@article {PS19,
    AUTHOR = {Przytycki, J\'ozef H. and Sikora, Adam S.},
     TITLE = {Skein algebras of surfaces},
   JOURNAL = {Trans. Amer. Math. Soc.},
  FJOURNAL = {Transactions of the American Mathematical Society},
    VOLUME = {371},
      YEAR = {2019},
    NUMBER = {2},
     PAGES = {1309--1332},
      ISSN = {0002-9947,1088-6850},
   MRCLASS = {57M27 (57M25)},
  MRNUMBER = {3885180},
MRREVIEWER = {Qi\ Chen},
       DOI = {10.1090/tran/7298},
       URL = {https://doi.org/10.1090/tran/7298},
}

@article {BW11,
    AUTHOR = {Bonahon, Francis and Wong, Helen},
     TITLE = {Quantum traces for representations of surface groups in {${\rm
              SL}_2(\Bbb C)$}},
   JOURNAL = {Geom. Topol.},
  FJOURNAL = {Geometry \& Topology},
    VOLUME = {15},
      YEAR = {2011},
    NUMBER = {3},
     PAGES = {1569--1615},
      ISSN = {1465-3060,1364-0380},
   MRCLASS = {57R56 (57Mxx)},
  MRNUMBER = {2851072},
       DOI = {10.2140/gt.2011.15.1569},
       URL = {https://doi.org/10.2140/gt.2011.15.1569},
}

@article {FK09,
    AUTHOR = {Fendley, Paul and Krushkal, Vyacheslav},
     TITLE = {Tutte chromatic identities from the {T}emperley-{L}ieb
              algebra},
   JOURNAL = {Geom. Topol.},
  FJOURNAL = {Geometry \& Topology},
    VOLUME = {13},
      YEAR = {2009},
    NUMBER = {2},
     PAGES = {709--741},
      ISSN = {1465-3060,1364-0380},
   MRCLASS = {57M15 (05C15 05C31 57R56)},
  MRNUMBER = {2469528},
MRREVIEWER = {Volodymyr\ Mazorchuk},
       DOI = {10.2140/gt.2009.13.709},
       URL = {https://doi.org/10.2140/gt.2009.13.709},
}

@article {FK10,
    AUTHOR = {Fendley, Paul and Krushkal, Vyacheslav},
     TITLE = {Link invariants, the chromatic polynomial and the {P}otts
              model},
   JOURNAL = {Adv. Theor. Math. Phys.},
  FJOURNAL = {Advances in Theoretical and Mathematical Physics},
    VOLUME = {14},
      YEAR = {2010},
    NUMBER = {2},
     PAGES = {507--540},
      ISSN = {1095-0761,1095-0753},
   MRCLASS = {57M27 (82B20)},
  MRNUMBER = {2721654},
MRREVIEWER = {Sergei\ K.\ Lando},
       DOI = {10.4310/atmp.2010.v14.n2.a4},
       URL = {https://doi.org/10.4310/atmp.2010.v14.n2.a4},
}

@article {FK23,
    AUTHOR = {Fendley, Paul and Krushkal, Vyacheslav},
     TITLE = {Topological quantum field theory and polynomial identities for
              graphs on the torus},
   JOURNAL = {Ann. Inst. Henri Poincar\'e{} D},
  FJOURNAL = {Annales de l'Institut Henri Poincar\'e{} D. Combinatorics,
              Physics and their Interactions},
    VOLUME = {10},
      YEAR = {2023},
    NUMBER = {2},
     PAGES = {277--298},
      ISSN = {2308-5827,2308-5835},
   MRCLASS = {57K16 (05C31 57M15 81T45 82B20)},
  MRNUMBER = {4581444},
MRREVIEWER = {Yilong\ Wang},
       DOI = {10.4171/aihpd/130},
       URL = {https://doi.org/10.4171/aihpd/130},
}

@article {AK16,
    AUTHOR = {Agol, Ian and Krushkal, Vyacheslav},
     TITLE = {Tutte relations, {TQFT}, and planarity of cubic graphs},
   JOURNAL = {Illinois J. Math.},
  FJOURNAL = {Illinois Journal of Mathematics},
    VOLUME = {60},
      YEAR = {2016},
    NUMBER = {1},
     PAGES = {273--288},
      ISSN = {0019-2082,1945-6581},
   MRCLASS = {05C31 (05C10 57M15 57R56)},
  MRNUMBER = {3665181},
MRREVIEWER = {Volodymyr\ Mazorchuk},
       URL = {http://projecteuclid.org/euclid.ijm/1498032033},
}

@article {Chb12,
    AUTHOR = {Chbili, Nafaa},
     TITLE = {Graph skein modules and symmetries of spatial graphs},
   JOURNAL = {J. Knot Theory Ramifications},
  FJOURNAL = {Journal of Knot Theory and its Ramifications},
    VOLUME = {21},
      YEAR = {2012},
    NUMBER = {9},
     PAGES = {1250090, 16},
      ISSN = {0218-2165,1793-6527},
   MRCLASS = {57M25 (05C10)},
  MRNUMBER = {2926573},
MRREVIEWER = {Jianyuan\ K.\ Zhong},
       DOI = {10.1142/S0218216512500903},
       URL = {https://doi.org/10.1142/S0218216512500903},
}

@article {Chb06,
    AUTHOR = {Chbili, Nafaa},
     TITLE = {Skein algebras of the solid torus and symmetric spatial
              graphs},
   JOURNAL = {Fund. Math.},
  FJOURNAL = {Fundamenta Mathematicae},
    VOLUME = {190},
      YEAR = {2006},
     PAGES = {1--10},
      ISSN = {0016-2736,1730-6329},
   MRCLASS = {57M27 (05C10 57M15 57M25)},
  MRNUMBER = {2232853},
MRREVIEWER = {Carlo\ Petronio},
       DOI = {10.4064/fm190-0-1},
       URL = {https://doi.org/10.4064/fm190-0-1},
}

@article {Dew73,
    AUTHOR = {Dewdney, A. K.},
     TITLE = {Wagner's theorem for torus graphs},
   JOURNAL = {Discrete Math.},
  FJOURNAL = {Discrete Mathematics},
    VOLUME = {4},
      YEAR = {1973},
     PAGES = {139--149},
      ISSN = {0012-365X,1872-681X},
   MRCLASS = {05C10},
  MRNUMBER = {309773},
MRREVIEWER = {E.\ A.\ Nordhaus},
       DOI = {10.1016/0012-365X(73)90076-9},
       URL = {https://doi.org/10.1016/0012-365X(73)90076-9},
}

@article {Neg94,
    AUTHOR = {Negami, Seiya},
     TITLE = {Diagonal flips in triangulations of surfaces},
   JOURNAL = {Discrete Math.},
  FJOURNAL = {Discrete Mathematics},
    VOLUME = {135},
      YEAR = {1994},
    NUMBER = {1--3},
     PAGES = {225--232},
      ISSN = {0012-365X,1872-681X},
   MRCLASS = {05C10},
  MRNUMBER = {1310882},
MRREVIEWER = {D.\ Barnette},
       DOI = {10.1016/0012-365X(93)E0101-9},
       URL = {https://doi.org/10.1016/0012-365X(93)E0101-9},
}

@article {FG00,
    AUTHOR = {Frohman, Charles and Gelca, R\u{a}zvan},
     TITLE = {Skein modules and the noncommutative torus},
   JOURNAL = {Trans. Amer. Math. Soc.},
  FJOURNAL = {Transactions of the American Mathematical Society},
    VOLUME = {352},
      YEAR = {2000},
    NUMBER = {10},
     PAGES = {4877--4888},
      ISSN = {0002-9947,1088-6850},
   MRCLASS = {57M25 (46L87 57M27 58B34)},
  MRNUMBER = {1675190},
MRREVIEWER = {Justin\ D.\ Roberts},
       DOI = {10.1090/S0002-9947-00-02512-5},
       URL = {https://doi.org/10.1090/S0002-9947-00-02512-5},
}

@misc{DKS24,
      title={Skein modules and character varieties of Seifert manifolds}, 
      author={Renaud Detcherry and Efstratia Kalfagianni and Adam S. Sikora},
      year={2025},
      eprint={2405.18557},
      archivePrefix={arXiv},
      primaryClass={math.GT},
      url={https://arxiv.org/abs/2405.18557},
      note={arxiv:2405.18557, to appear in Trans. Amer. Math. Soc.}
}

@misc{Jor23,
      title={Langlands duality for skein modules of 3-manifolds}, 
      author={David Jordan},
      year={2023},
      eprint={2302.14734},
      archivePrefix={arXiv},
      primaryClass={math.QA},
      url={https://arxiv.org/abs/2302.14734},
      note={arxiv:2302.14734}
}

@article {SW07,
    AUTHOR = {Sikora, Adam S. and Westbury, Bruce W.},
     TITLE = {Confluence theory for graphs},
   JOURNAL = {Algebr. Geom. Topol.},
  FJOURNAL = {Algebraic \& Geometric Topology},
    VOLUME = {7},
      YEAR = {2007},
     PAGES = {439--478},
      ISSN = {1472-2747,1472-2739},
   MRCLASS = {57M15 (05C10 05C30 17B20)},
  MRNUMBER = {2308953},
MRREVIEWER = {S.\ V.\ Duzhin},
       DOI = {10.2140/agt.2007.7.439},
       URL = {https://doi.org/10.2140/agt.2007.7.439},
}

@article {Kup96,
    AUTHOR = {Kuperberg, Greg},
     TITLE = {Spiders for rank {$2$} {L}ie algebras},
   JOURNAL = {Comm. Math. Phys.},
  FJOURNAL = {Communications in Mathematical Physics},
    VOLUME = {180},
      YEAR = {1996},
    NUMBER = {1},
     PAGES = {109--151},
      ISSN = {0010-3616,1432-0916},
   MRCLASS = {17B10 (22E60 81R05)},
  MRNUMBER = {1403861},
MRREVIEWER = {Stefano\ Capparelli},
       URL = {http://projecteuclid.org/euclid.cmp/1104287237},
}

@article {FKBL19,
    AUTHOR = {Frohman, Charles and Kania-Bartoszynska, Joanna and L\^e,
              Thang},
     TITLE = {Unicity for representations of the {K}auffman bracket skein
              algebra},
   JOURNAL = {Invent. Math.},
  FJOURNAL = {Inventiones Mathematicae},
    VOLUME = {215},
      YEAR = {2019},
    NUMBER = {2},
     PAGES = {609--650},
      ISSN = {0020-9910,1432-1297},
   MRCLASS = {57M27 (16H05 16S99)},
  MRNUMBER = {3910071},
MRREVIEWER = {Dylan\ G. L. Allegretti},
       DOI = {10.1007/s00222-018-0833-x},
       URL = {https://doi.org/10.1007/s00222-018-0833-x},
}

@misc{HSSSSW26,
      title={$\rm{PGL}_2$ webs at roots of unity}, 
      author={Higgins, Vijay and Silva, Bryan and Somers, Amy and Sundar, Srijan and Stephens, David and Wokhlu, Parth},
      year={2026},
      note={in preparation}
}

\end{document}